\documentclass[english]{article}
\usepackage[T1]{fontenc}
\usepackage[latin9]{inputenc}
\usepackage{float}
\usepackage{units}
\usepackage{bm}
\usepackage{dsfont}
\usepackage{amsmath}
\usepackage{amsthm}
\usepackage{amssymb}
\usepackage[authoryear,round]{natbib}

\makeatletter

\newcommand{\lyxmathsym}[1]{\ifmmode\begingroup\def\b@ld{bold}
  \text{\ifx\math@version\b@ld\bfseries\fi#1}\endgroup\else#1\fi}

\floatstyle{ruled}
\newfloat{algorithm}{tbp}{loa}
\providecommand{\algorithmname}{Algorithm}
\floatname{algorithm}{\protect\algorithmname}

\theoremstyle{plain}
\newtheorem{thm}{\protect\theoremname}[section]
\theoremstyle{plain}
\newtheorem{assumption}[thm]{\protect\assumptionname}
\theoremstyle{plain}
\newtheorem{lem}[thm]{\protect\lemmaname}
\theoremstyle{remark}
\newtheorem{rem}[thm]{\protect\remarkname}
\theoremstyle{definition}
\newtheorem{example}[thm]{\protect\examplename}

\usepackage{iclr2025_conference,times}

\usepackage{hyperref}
\usepackage{url}
\usepackage{cleveref}

\author{Zijian Liu\thanks{Corresponding author.} ~\& Zhengyuan Zhou\\
Stern School of Business, New York University\\
\texttt{\{zl3067,zzhou\}@stern.nyu.edu}}

\iclrfinalcopy

\makeatother

\usepackage{babel}
\providecommand{\assumptionname}{Assumption}
\providecommand{\examplename}{Example}
\providecommand{\lemmaname}{Lemma}
\providecommand{\remarkname}{Remark}
\providecommand{\theoremname}{Theorem}

\begin{document}
\global\long\def\E{\mathbb{\mathbb{E}}}%
\global\long\def\F{\mathcal{F}}%
\global\long\def\N{\mathbb{\mathbb{N}}}%
\global\long\def\R{\mathbb{R}}%
\global\long\def\B{\mathbb{B}}%
\global\long\def\argmin{\mathbb{\mathrm{argmin}}}%
\global\long\def\prog{\mathbb{\mathrm{prog}}}%
\global\long\def\O{\mathcal{O}}%
\global\long\def\1{\mathds{1}}%
\global\long\def\bx{\boldsymbol{x}}%
\global\long\def\by{\boldsymbol{y}}%
\global\long\def\bz{\boldsymbol{z}}%
\global\long\def\bg{\boldsymbol{g}}%
\global\long\def\bu{\boldsymbol{u}}%
\global\long\def\bm{\boldsymbol{m}}%
\global\long\def\bv{\boldsymbol{v}}%
\global\long\def\bw{\boldsymbol{w}}%
\global\long\def\bxi{\boldsymbol{\xi}}%
\global\long\def\bep{\boldsymbol{\epsilon}}%
\global\long\def\bzero{\boldsymbol{0}}%
\global\long\def\defeq{\triangleq}%
\global\long\def\p{\mathfrak{p}}%
\global\long\def\d{\mathrm{d}}%
\global\long\def\bD{\boldsymbol{D}}%

\title{Nonconvex Stochastic Optimization under Heavy-Tailed Noises: Optimal
Convergence without Gradient Clipping}
\maketitle
\begin{abstract}
Recently, the study of heavy-tailed noises in first-order nonconvex
stochastic optimization has gotten a lot of attention since it was
recognized as a more realistic condition as suggested by many empirical
observations. Specifically, the stochastic noise (the difference between
the stochastic and true gradient) is considered to have only a finite
$\mathfrak{p}$-th moment where $\mathfrak{p}\in\left(1,2\right]$
instead of assuming it always satisfies the classical finite variance
assumption. To deal with this more challenging setting, people have
proposed different algorithms and proved them to converge at an optimal
$\mathcal{O}(T^{\frac{1-\mathfrak{p}}{3\mathfrak{p}-2}})$ rate for
smooth objectives after $T$ iterations. Notably, all these new-designed
algorithms are based on the same technique -- gradient clipping.
Naturally, one may want to know whether the clipping method is a necessary
ingredient and the only way to guarantee convergence under heavy-tailed
noises. In this work, by revisiting the existing Batched Normalized
Stochastic Gradient Descent with Momentum (Batched NSGDM) algorithm,
we provide the first convergence result under heavy-tailed noises
but \textit{without} gradient clipping. Concretely, we prove that
Batched NSGDM can achieve the optimal $\mathcal{O}(T^{\frac{1-\mathfrak{p}}{3\mathfrak{p}-2}})$
rate even under the relaxed smooth condition. More interestingly,
we also establish the first $\mathcal{O}(T^{\frac{1-\mathfrak{p}}{2\mathfrak{p}}})$
convergence rate in the case where the tail index $\mathfrak{p}$
is unknown in advance, which is arguably the common scenario in practice.
\end{abstract}

\section{Introduction}

This paper studies the optimization problem $\min_{\bx\in\R^{d}}F(\bx)$
where $F:\R^{d}\to\R$ is differentiable and could be nonconvex. When
$F$ is smooth (i.e., the gradient of $F$ is Lipschitz), the classical
first-order method, Gradient Descent (GD), is known to converge at
the optimal rate $\O(T^{-\frac{1}{2}})$ to find a stationary point
(i.e., to minimize $\left\Vert \nabla F(\bx)\right\Vert $) \citep{nesterov2018lectures}.
However, a main drawback of GD in the modern view is that it requires
true gradients, which could be computationally burdensome (e.g., large-scale
tasks) or even infeasible to obtain (e.g., streaming data). As such,
a famous variant of GD, Stochastic Gradient Descent (SGD) \citep{robbins1951stochastic},
has become the gold standard and been widely implemented nowadays
due to its lightweight yet efficient computational procedure. Under
the standard finite variance noise condition, i.e., the second moment
of the difference between the stochastic gradient and the true gradient
is bounded, SGD has been proved to converge in the rate of $\O(T^{-\frac{1}{4}})$
\citep{ghadimi2013stochastic}, which is unimprovable if without further
assumptions as indicated by the lower bound \citep{arjevani2023lower}.

Although the finite variance assumption has been widely adopted in
theoretical study (see, e.g., \citet{lan2020first}), it has been
recently recognized as too optimistic in modern machine learning tasks
pointed out by empirical observations \citep{simsekli2019tail,csimcsekli2019heavy,zhang2020adaptive},
which reveal a more realistic setting: the heavy-tailed regime, i.e.,
the stochastic noise only has a finite $\p$-th moment where $\mathfrak{p}\in\left(1,2\right]$.
Such a new assumption brings new challenges in both algorithmic design
and theoretical analysis since SGD unfortunately fails to work and
the prior theory of SGD becomes invalid when $\p<2$ \citep{zhang2020adaptive}.

To resolve the failure of SGD under heavy-tailed noises, Clipped SGD
(or its further variants) has been proposed and shown to converge
in both expectation \citep{zhang2020adaptive} and high probability
\citep{cutkosky2021high,pmlr-v195-liu23c,nguyen2023improved,pmlr-v235-liu24bo}
at a rate $\O(T^{\frac{1-\p}{3\p-2}})$, which is indeed optimal \citep{zhang2020adaptive}.
As suggested by the name, the central tool lying in all these existing
algorithms is gradient clipping, which seems to be a necessary ingredient
against heavy-tailed noises so far. Therefore, we are naturally led
to the following question:
\begin{center}
\textit{Is gradient clipping the only way to guarantee convergence
under heavy-tailed noises? If not, can we find an algorithm that converges
at the optimal rate $\O(T^{\frac{1-\p}{3\p-2}})$ without clipping?}
\par\end{center}

Another important but often omitted issue in previous studies is that
the tail index $\p$ of heavy-tailed noises is always implicitly assumed
to be known and then used to choose the clipping magnitude and the
stepsize (or the momentum parameter) \citep{zhang2020adaptive,cutkosky2021high,pmlr-v195-liu23c,nguyen2023improved,pmlr-v235-liu24bo}.
However, knowing $\p$ exactly or even estimating its approximate
value is a non-trivial task in lots of cases, e.g., the online setting.
Consequently, the existing convergence theory for Clipped SGD immediately
becomes vacuous when no prior information on $\p$ is guaranteed,
which is however arguably the common scenario in practice. The above
discussion thereby leads us to another research question:
\begin{center}
\textit{Does there exist an algorithm that provably converges under
heavy-tailed noises even if the tail index $\p$ is unknown?}
\par\end{center}

This work provides affirmative answers to both of the above questions
by revisiting a well-known technique in optimization: gradient normalization,
which is surprisingly effective even under heavy-tailed noises as
demonstrated by our refined theoretical analysis.

\subsection{Our Contributions}

We study the Batched Normalized Stochastic Gradient Descent with Momentum
(Batched NSGDM) algorithm \citep{pmlr-v119-cutkosky20b} under heavy-tailed
noises and establish several new results:
\begin{itemize}
\item We prove that Batched NSGDM converges in expectation at the optimal
rate $\O(T^{\frac{1-\p}{3\p-2}})$ when noises only have a finite
$\p$-th moment where $\mathfrak{p}\in\left(1,2\right]$, which is
the first convergence result under heavy-tailed noises not requiring
gradient clipping\footnote{In preparing the camera-ready version of the paper, it was brought
to our attention that two independent and concurrent works established
similar results \citep{hubler2024gradient,sun2024gradient}.}.
\item We establish a refined lower complexity result having a more precise
order on problem-dependent parameters (e.g., the noise level $\sigma_{0}$),
which perfectly matches our new convergence theory for Batched NSGDM
further showing the optimality of our analysis.
\item We initiate the study of optimization under heavy-tailed noises with
an unknown tail index $\p$ and provide the first provable rate $\O(T^{\frac{1-\p}{2\p}})$
also achieved by the same Batched NSGDM method indicating the robustness
of gradient normalization against heavy-tailed noises\footnote{\citet{hubler2024gradient} also gave a matched lower bound for Batched
NSGD, i.e., $\beta_{t}\equiv0$ in Algorithm \ref{alg:NSGDM}.}.
\item Our analysis goes beyond the classical smoothness condition and heavy-tailed
noises assumption studied in previous works and is the first to hold
under their generalized counterparts (see Assumptions \ref{assu:smooth}
and \ref{assu:general-heavy-noise}).
\item Our proof is based on a novel expected inequality for the vector-valued
martingale difference sequence, which might be of independent interest.
\end{itemize}

\subsection{Related Work\label{subsec:related-work}}

We focus on the literature on nonconvex problems under heavy-tailed
noises. For recent progress on convex optimization, the reader can
refer to \citet{zhang2022parameter,pmlr-v202-sadiev23a,liu2023stochastic,NEURIPS2023_ca24eb48,puchkin2024breaking,pmlr-v235-gorbunov24a}
for details.

\textbf{Upper bound under heavy-tailed noises. }For smooth objectives,
different works have established the optimal rate $\O(T^{\frac{1-\p}{3\p-2}})$
(up to extra logarithmic factors) for Clipped SGD or its further variants
\citep{zhang2020adaptive,cutkosky2021high,pmlr-v195-liu23c,nguyen2023improved,pmlr-v235-liu24bo},
among which, \citet{zhang2020adaptive} and \citet{nguyen2023improved}
respectively provide the best expected and high-probability bounds
for Clipped SGD as their results do not contain any extra $\O(\log T)$
factor. Notably, \citet{zhang2020adaptive} can recover the standard
$\O(T^{-\frac{1}{2}})$ rate in the noiseless case. In contrast, the
rates by \citet{cutkosky2021high,pmlr-v195-liu23c,nguyen2023improved,pmlr-v235-liu24bo}
are not adaptive to the noise level. In addition, we note that the
results from \citet{cutkosky2021high} also depend on an extra vulnerable
assumption, i.e., the stochastic gradient itself has a finite $\p$-th
moment, which can be easily violated. However, all these works require
a known tail index $\p$ to establish convergence theory, which may
not be realistic in practice.

\textbf{Lower bound under heavy-tailed noises.} As far as we know,
\citet{zhang2020adaptive} is the first and only work that provides
the $\Omega(T^{\frac{1-\p}{3\p-2}})$ lower bound for nonconvex optimization
under the classical smooth assumption and the finite $\p$-th moment
condition on noises, where the proof is based on the tool named probability
zero-chain developed by \citet{arjevani2023lower}. However, we should
remind the reader that the lower bound by \citet{zhang2020adaptive}
is actually proved based on assuming a finite $\p$-th moment on the
stochastic gradient instead of the noise and hence may fail to provide
a correct dependence on problem-dependent parameters like the noise
level $\sigma_{0}$.

In addition, we provide a quick review of the gradient normalization
technique.

\textbf{Gradient normalization.} The normalized gradient method has
a long history and could date back to the pioneering work of \citet{nesterov1984minimization},
which is the first paper to suggest considering normalization in (quasi-)convex
optimization problems and provides a theoretical convergence rate.
Many later works (e.g., \citet{kiwiel2001convergence,hazan2015beyond,levy2016power,nacson2019convergence})
further explore the potential of gradient normalization. In deep learning,
gradient normalization (or its variant) has also gained more and more
attention since it can tackle the gradient explosion/vanish issue
and has been observed to accelerate the training \citep{you2017large,you2019large}.
However, a provable theory still lacks for general nonconvex problems
until \citep{pmlr-v119-cutkosky20b}, who established the first meaningful
bound by adding momentum to the normalized gradient method.

\section{Preliminaries\label{sec:Preliminaries}}

\textbf{Notation.} $\N$ denotes the set of natural numbers (excluding
$0$). $\left[T\right]\defeq\left\{ 1,2,\cdots,T\right\} $ for any
$T\in\N$. $\left\langle \cdot,\cdot\right\rangle $ is the Euclidean
inner product on $\R^{d}$ and $\left\Vert \cdot\right\Vert \defeq\sqrt{\left\langle \cdot,\cdot\right\rangle }$
is the $\ell_{2}$ norm. Given a sequence $r_{t}\in\R,\forall t\in\left[T\right]$,
we use the notation $r_{s:t}\defeq\prod_{\ell=s}^{t}r_{\ell}$ for
any $1\leq s\leq t\le T$ and $r_{s:t}\defeq1$ if $s>t$.

We consider the optimization problem $\min_{\bx\in\R^{d}}F(\bx)$
in this work, where $F:\R^{d}\to\R$ is differentiable and potentially
nonconvex. A computationally tractable goal is to minimize $\left\Vert \nabla F(\bx)\right\Vert $,
which is the focus hereinafter. Our analysis relies on the following
assumptions.
\begin{assumption}
\textbf{\label{assu:lb}Finite Lower Bound.} $F_{*}\defeq\inf_{\bx\in\R^{d}}F(\bx)>-\infty$.
\end{assumption}

\begin{assumption}
\textbf{\label{assu:smooth}Generalized Smoothness.} There exist $L_{0}\geq0$
and $L_{1}\geq0$ such that $\left\Vert \nabla F(\bx)-\nabla F(\by)\right\Vert \leq\left(L_{0}+L_{1}\left\Vert \nabla F(\bx)\right\Vert \right)\left\Vert \bx-\by\right\Vert $
for any $\bx,\by\in\R^{d}$ satisfying $\left\Vert \bx-\by\right\Vert \leq\frac{1}{L_{1}}.$
\end{assumption}

The original definition of generalized smoothness was proposed by
\citet{Zhang2020Why} but required the objective to be twice differentiable.
Here, we instead adopt a weaker version introduced later by \citet{zhang2020improved},
which only needs $F$ to be differentiable. As one can see, Assumption
\ref{assu:smooth} degenerates to the standard $L_{0}$-smoothness
when $L_{1}=0$ and hence is more general. We note that there exist
other versions of generalized smoothness proposed recently and refer
to \citet{pmlr-v202-chen23ar,NEURIPS2023_7e8bb8d1,NEURIPS2023_a3cc5012}
for details.
\begin{assumption}
\textbf{\label{assu:unbias}Unbiased Estimator.} At the $t$-th iteration,
we can access a batch of unbiased gradient estimator $G_{t}\defeq\left\{ \bg_{t}^{1},\cdots,\bg_{t}^{B}\right\} $,
i.e., $\E\left[\bg_{t}^{i}\mid\F_{t-1}\right]=\nabla F(\bx_{t}),\forall i\in\left[B\right]$,
where $B$ is the batch size and $\F_{t}\defeq\sigma(G_{1},\cdots,G_{t})$
denotes the natural filtration. Moreover, we assume that $\bg_{t}^{i},\forall i\in\left[B\right]$
are mutually independent for any fixed $t$.
\end{assumption}

\begin{assumption}
\textbf{\label{assu:general-heavy-noise}Generalized Heavy-Tailed
Noises.} There exist $\p\in\left(1,2\right]$, $\sigma_{0}\geq0$
and $\sigma_{1}\geq0$ such that $\E\left[\left\Vert \bxi_{t}^{i}\right\Vert ^{\p}\mid\F_{t-1}\right]\leq\sigma_{0}^{\p}+\sigma_{1}^{\p}\left\Vert \nabla F(\bx_{t})\right\Vert ^{\p},\forall i\in\left[B\right]$
almost surely where $\bxi_{t}^{i}\defeq\bg_{t}^{i}-\nabla F(\bx_{t})$.
\end{assumption}

We remark that Assumption \ref{assu:general-heavy-noise} is a relaxation
of the traditional heavy-tailed noises assumption (i.e., set $\sigma_{1}=0$)
and is new as far as we know. The reason for proposing this generalized
version is that the finite $\p$-th moment requirement used in prior
works can be violated in certain situations where the new assumption
instead holds. We refer the reader to Example \ref{exa:noises} provided
in Appendix \ref{sec:example} for details. In addition, such a form
of Assumption \ref{assu:general-heavy-noise} may reminisce about
the affine variance condition studied in the existing literature \citep{bottou2018optimization}.
Indeed, this new assumption is inspired by it and can also be viewed
as its extension.

However, to make our work more reader-friendly, we will stick to the
case $\sigma_{1}=0$ (i.e., the classical heavy-tailed noises assumption
used in prior works) in the main text to focus on conveying our high-level
idea and avoid further complicating the analysis. The full version
of our new result considering an arbitrary pair $(\sigma_{0},\sigma_{1})$
in Assumption \ref{assu:general-heavy-noise} is deferred to Appendix
\ref{sec:Full-Theorems}.

To finish this section, we introduce the following smoothness inequality
under Assumption \ref{assu:smooth}, whose proof is omitted and can
be found in, for example, \citet{zhang2020improved}.
\begin{lem}
\label{lem:smooth-ineq}Under Assumption \ref{assu:smooth}, for any
$\bx,\by\in\R^{d}$ satisfying\textup{ $\left\Vert \bx-\by\right\Vert \leq\frac{1}{L_{1}}$},
there is
\[
F(\by)\leq F(\bx)+\langle\nabla F(\bx),\by-\bx\rangle+\frac{L_{0}+L_{1}\left\Vert \nabla F(\bx)\right\Vert }{2}\left\Vert \bx-\by\right\Vert ^{2}.
\]
\end{lem}

\section{Convergence without Gradient Clipping\label{sec: algo}}

\begin{algorithm}[H]
\caption{\label{alg:NSGDM}Batched Normalized Stochastic Gradient Descent with
Momentum (Batched NSGDM)}

\textbf{Input:} initial point $\bx_{1}\in\R^{d}$, batch size $B\in\N$,
momentum parameter $\beta_{t}\in\left[0,1\right]$, stepsize $\eta_{t}>0$

\textbf{for} $t=1$ \textbf{to} $T$ \textbf{do}

$\quad$$\bg_{t}=\frac{1}{B}\sum_{i=1}^{B}\bg_{t}^{i}$

$\quad$$\bm{_{t}}=\beta_{t}\bm{_{t-1}}+(1-\beta_{t})\bg_{t}$\hfill{}$\triangleright\text{ where }\bm{_{0}}\defeq\bg_{1}$

$\quad$$\bx_{t+1}=\bx_{t}-\eta_{t}\frac{\bm{_{t}}}{\left\Vert \bm{_{t}}\right\Vert }$\hfill{}$\triangleright\text{ where }\frac{\bzero}{\left\Vert \bzero\right\Vert }\defeq\bzero$

\textbf{end for}
\end{algorithm}

\begin{rem}
Instead of the norm normalization in Algorithm \ref{alg:NSGDM},
we can consider an elementwise normalization rule, i.e., $\bx_{t+1}[i]=\bx_{t}[i]-\eta_{t,i}\frac{\bm{_{t}}[i]}{\left|\bm{_{t}}[i]\right|},\forall i\in\left[d\right]$,
which is known as Batched Sign Stochastic Gradient Descent with Momentum
(Batched SSGDM). By applying the idea introduced in Section \ref{sec:analysis},
one can also establish the convergence of Batched SSGDM under heavy-tailed
noises.
\end{rem}

The method we are interested in, Batched NSGDM \citep{pmlr-v119-cutkosky20b},
is provided above in Algorithm \ref{alg:NSGDM}. Compared to the widely
used Batched SGDM algorithm, the only difference is the extra normalization
step, which we will show has a crucial effect when dealing with heavy-tailed
noises. Since many prior works (e.g., \citet{you2017large,pmlr-v119-cutkosky20b,jin2021non})
have explained how and why Batched NSGDM works (in the finite/affine
variance case), we hence do not repeat the discussion here again.
The reader seeking the intuition behind the algorithm could refer
to, for example, \citet{pmlr-v119-cutkosky20b} for details.

Now we are ready to present our new result for this classical algorithm
under heavy-tailed noises.

\subsection{Convergence with A Known Tail Index $\protect\p$\label{subsec:full}}

In this subsection, we provide the convergence rate of Batched NSGDM
under an ideal situation, i.e., when every problem-dependent parameter
is known, which is commonly assumed implicitly in the optimization
literature.
\begin{thm}
\label{thm:main-fixed-full-info}Under Assumptions \ref{assu:lb},
\ref{assu:smooth}, \ref{assu:unbias} and \ref{assu:general-heavy-noise}
(with $\sigma_{1}=0$), let $\Delta_{1}\defeq F(\bx_{1})-F_{*}$,
then for any $T\in\N$, by taking
\begin{align*}
\beta_{t} & \equiv\beta=1-\min\left\{ 1,\max\left\{ \left(\frac{\Delta_{1}L_{1}+\sigma_{0}}{\sigma_{0}T}\right)^{\frac{\p}{2\p-1}},\left(\frac{\Delta_{1}L_{0}}{\sigma_{0}^{2}T}\right)^{\frac{\p}{3\p-2}}\right\} \right\} ,\\
\eta_{t} & \equiv\eta=\min\left\{ \sqrt{\frac{(1-\beta)\Delta_{1}}{L_{0}T}},\frac{1-\beta}{8L_{1}}\right\} ,\quad B=1,
\end{align*}
Algorithm \ref{alg:NSGDM} guarantees
\[
\frac{1}{T}\sum_{t=1}^{T}\E\left[\left\Vert \nabla F(\bx_{t})\right\Vert \right]=\O\left(\frac{\Delta_{1}L_{1}}{T}+\sqrt{\frac{\Delta_{1}L_{0}}{T}}+\frac{(\Delta_{1}L_{1})^{\frac{\p-1}{2\p-1}}\sigma_{0}^{\frac{\p}{2\p-1}}+\sigma_{0}}{T^{\frac{\p-1}{2\p-1}}}+\frac{(\Delta_{1}L_{0})^{\frac{\p-1}{3\p-2}}\sigma_{0}^{\frac{\p}{3\p-2}}}{T^{\frac{\p-1}{3\p-2}}}\right).
\]
\end{thm}

Theorem \ref{thm:main-fixed-full-info} states the convergence rate
of Algorithm \ref{alg:NSGDM} under heavy-tailed noises. The full
version, Theorem \ref{thm:fixed-full-info}, that works for any $\sigma_{1}\geq0$
is deferred to Appendix \ref{sec:Full-Theorems}. As a quick sanity
check, Theorem \ref{thm:main-fixed-full-info} reduces to the rate
$\O((\Delta_{1}L_{0}\sigma_{0}^{2}/T)^{\frac{1}{4}})$ when $\p=2$
(i.e., the finite variance case), which is known to be tight not only
in $T$ but also in $\Delta_{1},L_{0}$ and $\sigma_{0}$ \citep{arjevani2023lower}.

We would like to discuss this theorem here further. First and most
importantly, as far as we know, this is the first and only convergence
result for nonconvex stochastic optimization under heavy-tailed noises
but without employing the gradient clipping technique, which is the
central tool for all previous algorithms under the same setting, not
to mention the rate $\O(T^{\frac{1-\p}{3\p-2}})$ is also optimal
in $T$ since it matches the lower bound $\Omega(T^{\frac{1-\p}{3\p-2}})$
proved in \citet{zhang2020adaptive}. In fact, the lower-order term
$\O((\Delta_{1}L_{0}\sigma_{0}^{\frac{\p}{\p-1}}/T)^{\frac{\p-1}{3\p-2}})$
in our rate is also tight in $\Delta_{1},L_{0}$ and $\sigma_{0}$
as indicated by Theorem \ref{thm:main-lower-bound} below, where we
present a refined lower bound by extending the prior result \citep{zhang2020adaptive}.
To the best of our knowledge, we are also the first to obtain a tight
dependence on these parameters compared to the previous best-known
bounds for Clipped SGD that are only optimal in $T$ \citep{zhang2020adaptive,nguyen2023improved}.
The proof of Theorem \ref{thm:main-lower-bound} is given in Appendix
\ref{sec:Full-Theorems} and mostly follows the same way established
in \citet{carmon2020lower,arjevani2023lower,zhang2020adaptive} but
with a simple alteration to ensure that the noise magnitude $\sigma_{0}$
shows up in the correct order.
\begin{thm}
\label{thm:main-lower-bound}For any given $\p\in\left(1,2\right]$,
$\Delta_{1},L_{0},\sigma_{0}>0$, and small enough $\varepsilon>0$,
there exist a function $F$ (depending on the previous parameters)
satisfying Assumptions \ref{assu:lb}, \ref{assu:smooth} (with $L_{1}=0$)
and $F(\bzero)-F_{*}\leq\Delta_{1}$ and a stochastic oracle satisfying
Assumptions \ref{assu:unbias} (with $B=1$) and \ref{assu:general-heavy-noise}
(with $\sigma_{1}=0$) such that any zero-respecting algorithm\footnote{A first-order algorithm is called zero-respecting if it satisfies
$\bx_{t}\in\cup_{s<t}\mathrm{support}(\bg_{s}),\forall t\in\N$. The
reader could refer to Definition 1 in \citet{arjevani2023lower} for
details.} startinig from $\bx_{1}=\bzero$ requires $\Omega(\Delta_{1}L_{0}\sigma_{0}^{\frac{\p}{\p-1}}\varepsilon^{-\frac{3\p-2}{\p-1}})$
iterations to find an $\varepsilon$-stationary point, i.e., $\E\left[\left\Vert \nabla F(\bx)\right\Vert \right]\leq\varepsilon$.
\end{thm}

\begin{rem}
The careful reader may find that Theorem \ref{thm:main-lower-bound}
is established under the classical smooth assumption instead of the
relaxed $(L_{0},L_{1})$-smooth condition. However, noticing that
the function class satisfying the traditional smooth condition is
a subclass of the function set fulfilling generalized smoothness,
this lower bound can thus be directly applied to the setting considered
in Theorem \ref{thm:main-fixed-full-info}.
\end{rem}

\begin{rem}
We also note that, when $\p=2$, Theorem \ref{thm:main-lower-bound}
perfectly matches the existing lower bound $\Omega(\Delta_{1}L_{0}\sigma_{0}^{2}\varepsilon^{-4})$
in the finite variance case \citep{arjevani2023lower}.
\end{rem}

Moreover, we would like to mention that Theorem \ref{thm:main-fixed-full-info}
holds under generalized smoothness (and generalized heavy-tailed noises),
which greatly extends the implication of our result compared to the
previous works \citep{zhang2020adaptive,nguyen2023improved} that
can be only applied to the classical setting. In addition, the reader
may want to ask why we need a batch size $B$ given it is set to $1$.
The reason for considering $B$ is that it plays an important role
when $\sigma_{1}>0$. To be precise, $B$ could possibly be larger
than $1$ to guarantee the convergence when $\sigma_{1}>0$. For details
about how the batch size $B$ works, please refer to Theorem \ref{thm:fixed-full-info}
and its proof. Lastly, we want to point out that Theorem \ref{thm:main-fixed-full-info}
perfectly recovers the fastest rate in the noiseless case. In other
words, when $\sigma_{0}=0$, the rate degenerates to $\O(\Delta_{1}L_{1}/T+\sqrt{\Delta_{1}L_{0}/T})$,
which is the best result for deterministic $(L_{0},L_{1})$-smooth
nonconvex optimization as far as we are aware \citep{liu2023near}. 

Given the above discussion, we believe that our work provides a new
insight on how to deal with heavy-tailed noises, i.e., using normalization,
in stochastic nonconvex optimization problems.

\subsection{Convergence without A Known Tail Index $\protect\p$\label{subsec:partial}}

As pointed out at the beginning of Subsection \ref{subsec:full},
the convergence result shown in Theorem \ref{thm:main-fixed-full-info}
can only hold under full information because the choices of $\beta$
and $\eta$ heavily rely on problem-dependent parameters, which are,
however, hard to estimate in practice. Especially, assuming prior
information on $\p$ is not realistic. As such, we will try to reduce
the dependence on these parameters in this subsection.
\begin{thm}
\label{thm:main-fixed-partial-info}Under Assumptions \ref{assu:lb},
\ref{assu:smooth}, \ref{assu:unbias} and \ref{assu:general-heavy-noise}
(with $\sigma_{1}=0$), let $\Delta_{1}\defeq F(\bx_{1})-F_{*}$,
then for any $T\in\N$, by taking
\begin{eqnarray*}
\beta_{t}\equiv\beta=1-\frac{1}{T^{\frac{1}{2}}}, & \eta_{t}\equiv\eta=\min\left\{ \frac{1}{T^{\frac{3}{4}}},\frac{1}{8L_{1}T^{\frac{1}{2}}}\right\} , & B=1,
\end{eqnarray*}
Algorithm \ref{alg:NSGDM} guarantees
\[
\frac{1}{T}\sum_{t=1}^{T}\E\left[\left\Vert \nabla F(\bx_{t})\right\Vert \right]=\O\left(\frac{\Delta_{1}L_{1}}{\sqrt{T}}+\frac{\Delta_{1}+L_{0}}{T^{\frac{1}{4}}}+\frac{\sigma_{0}}{T^{\frac{\p-1}{2\p}}}\right).
\]
\end{thm}

Theorem \ref{thm:main-fixed-partial-info} shows the first provable
convergence upper bound $\O(T^{\frac{1-\p}{2\p}})$ when the tail
index $\p$ is unknown and, at the same time, tries to relax the dependence
on other parameters. As one can see, the only parameter that we require
now is $L_{1}$. It is noteworthy that, once the time horizon $T$
is large enough to satisfy $T=\Omega(L_{1}^{4})$ (or equivalently,
$L_{1}$ is small enough to satisfy $L_{1}=\O(T^{\frac{1}{4}})$),
we can get rid of $L_{1}$ in $\eta$ and thus do not need any information
on the problem. A remarkable implication is that, in the same classical
smooth case (i.e., when $L_{1}=0$) studied in previous works \citep{zhang2020adaptive,nguyen2023improved},
Batched NSGDM can converge in the rate $\O(T^{\frac{1-\p}{2\p}})$
without knowing any of $\p,\Delta_{1},L_{0},\sigma_{0}$. In contrast,
the choices of the clipping magnitude and the stepsize in Clipped
SGD provided by \citet{zhang2020adaptive,nguyen2023improved} heavily
depend on these parameters.

In addition, there are several points we would like to clarify. First,
one may want to ask whether the requirement of knowing $L_{1}$ can
be totally lifted, which we incline to a positive answer. But to
prevent deviating from the main topic of our paper --- heavy-tailed
noises, we defer the detailed discussion about $L_{1}$ to Appendix
\ref{sec:L1}, in which we will explain the reason and talk about
a possible way to achieve this goal. Another question that we view
important but currently have no answer to is whether the $\O(T^{\frac{1-\p}{3\p-2}})$
rate is still achievable when $\p$ is unknown, which we leave as
an interesting direction to be explored in the future. Moreover, the
reader may find that the rate in Theorem \ref{thm:main-fixed-partial-info}
loses the adaptivity on $\sigma_{0}$ since it can only guarantee
the $\O(T^{-\frac{1}{4}})$ convergence instead of the optimal $\O(T^{-\frac{1}{2}})$
rate in the noiseless case. We remark that this is a common phenomenon
for optimization algorithms when oblivious to the level of noise.
For example, it is well-known that SGD suffers the same issue even
under the finite variance assumption but when $\sigma_{0}$ is unknown.
The last thing we have to mention is that, unfortunately, the good
property of not needing the tail index $\p$ may fail when considering
$\sigma_{1}>0$ in the generalized heavy-tailed assumption. Precisely
speaking, we can only prove that there exists a constant threshold
$\sigma_{1}^{*}>0$ such that $B$ can always be set to $1$ if $\sigma_{1}\leq\sigma_{1}^{*}$
and $B$ has to be chosen based on $\p$ when $\sigma_{1}>\sigma_{1}^{*}$.
The details can be found in Theorem \ref{thm:fixed-partial-info}
in the appendix.

In summary, we exhibit the first convergence result under heavy-tailed
noises when only partial information about the problem is available.
Particularly, under the classical heavy-tailed noises, we show how
to guarantee convergence even if the tail index $\p$ is unknown.

\section{How Gradient Normalization Works\label{sec:analysis}}

In this section, we will explain how gradient normalization works
under heavy-tailed noises by both intuitive discussion and theoretical
analysis. Moreover, to keep the analysis simple and better compare
the difference between Batched NSGDM and Clipped SGD studied previously,
we will focus on the classical $L_{0}$-smooth case (i.e., take $L_{1}=0$
in Assumption \ref{assu:smooth}) under finite $\p$-th moment noises
(i.e., take $\sigma_{1}=0$ in Assumption \ref{assu:general-heavy-noise})
to align with prior works. Due to limited space, the missing proofs
of presented lemmas (and their full version) are deferred in the appendix.

\subsection{What does Gradient Clipping do?}

Before analyzing Algorithm \ref{alg:NSGDM}, let us first recap the
update rule of Clipped SGD:
\begin{equation}
\bx_{t+1}=\bx_{t}-\eta_{t}\hat{\bg}_{t},\quad\hat{\bg}_{t}\defeq\min\left\{ 1,\frac{\tau_{t}}{\left\Vert \bg_{t}\right\Vert }\right\} \bg_{t},\label{eq:clipped-sgd-1}
\end{equation}
where $\hat{\bg}_{t}$ is the clipped gradient and $\tau_{t}>0$ is
known as the clipping magnitude playing a critical role in the convergence.
Especially, (\ref{eq:clipped-sgd-1}) can recover SGD by setting $\tau_{t}=+\infty$.

We provide a simple analysis here to illustrate how gradient clipping
helps Clipped SGD to converge in expectation. First, by the well-known
smoothness inequality, i.e., Lemma \ref{lem:smooth-ineq} with $L_{1}=0$,
we know
\begin{align}
F(\bx_{t+1}) & \leq F(\bx_{t})+\left\langle \nabla F(\bx_{t}),\bx_{t+1}-\bx_{t}\right\rangle +\frac{L_{0}}{2}\left\Vert \bx_{t+1}-\bx_{t}\right\Vert ^{2}\nonumber \\
 & \overset{\eqref{eq:clipped-sgd-1}}{=}F(\bx_{t})-\eta_{t}\left\langle \nabla F(\bx_{t}),\hat{\bg}_{t}\right\rangle +\frac{\eta_{t}^{2}L_{0}}{2}\left\Vert \hat{\bg}_{t}\right\Vert ^{2}\nonumber \\
 & =F(\bx_{t})-\left(\eta_{t}-\frac{\eta_{t}^{2}L_{0}}{2}\right)\left\Vert \nabla F(\bx_{t})\right\Vert ^{2}-(\eta_{t}-\eta_{t}^{2}L_{0})\left\langle \nabla F(\bx_{t}),\bep_{t}\right\rangle +\frac{\eta_{t}^{2}L_{0}}{2}\left\Vert \bep_{t}\right\Vert ^{2},\label{eq:clipped-sgd-2}
\end{align}
where $\bep_{t}\defeq\hat{\bg}_{t}-\nabla F(\bx_{t})$. Following
the existing literature (e.g., \citet{cutkosky2021high,pmlr-v195-liu23c}),
we next decompose $\bep_{t}$ into $\bep_{t}=\bep_{t}^{u}+\bep_{t}^{b}$,
where $\bep_{t}^{u}\defeq\hat{\bg}_{t}-\E\left[\hat{\bg}_{t}\mid\F_{t-1}\right]$
and $\bep_{t}^{b}\defeq\E\left[\hat{\bg}_{t}\mid\F_{t-1}\right]-\nabla F(\bx_{t})$.
By noticing $\E\left[\left\langle \nabla F(\bx_{t}),\bep_{t}^{u}\right\rangle \right]=\E\left[\left\langle \bep_{t}^{b},\bep_{t}^{u}\right\rangle \right]=0$,
we thus have
\begin{align}
\E\left[F(\bx_{t+1})\right]\leq & \E\left[F(\bx_{t})\right]-\left(\eta_{t}-\frac{\eta_{t}^{2}L_{0}}{2}\right)\E\left[\left\Vert \nabla F(\bx_{t})\right\Vert ^{2}\right]\nonumber \\
 & +(\eta_{t}-\eta_{t}^{2}L_{0})\E\left[\left\langle \nabla F(\bx_{t}),-\bep_{t}^{b}\right\rangle \right]+\frac{\eta_{t}^{2}L_{0}}{2}\E\left[\left\Vert \bep_{t}^{u}\right\Vert ^{2}+\left\Vert \bep_{t}^{b}\right\Vert ^{2}\right]\nonumber \\
\overset{\text{if }\eta_{t}\leq\nicefrac{1}{L_{0}}}{\leq} & \E\left[F(\bx_{t})\right]-\frac{\eta_{t}}{2}\E\left[\left\Vert \nabla F(\bx_{t})\right\Vert ^{2}\right]+\frac{\eta_{t}^{2}L_{0}}{2}\E\left[\left\Vert \bep_{t}^{u}\right\Vert ^{2}\right]+\frac{\eta_{t}}{2}\E\left[\left\Vert \bep_{t}^{b}\right\Vert ^{2}\right],\label{eq:clipped-sgd-3}
\end{align}
where the last step is by $\E\left[\left\langle \nabla F(\bx_{t}),-\bep_{t}^{b}\right\rangle \right]\leq\frac{\E\left[\left\Vert \nabla F(\bx_{t})\right\Vert ^{2}\right]+\E\left[\left\Vert \bep_{t}^{b}\right\Vert ^{2}\right]}{2}$
and $\eta_{t}-\eta_{t}^{2}L_{0}\geq0$.
\begin{itemize}
\item Now let us check what will happen for SGD if without clipping, i.e.,
taking $\tau_{t}=+\infty$. In this case, (\ref{eq:clipped-sgd-3})
recovers the classical one-step inequality for SGD by observing $\bep_{t}^{u}=\bxi_{t}\defeq\bg_{t}-\nabla F(\bx_{t})$
and $\bep_{t}^{b}=\bzero$ now, which also reveals why SGD may fail
to converge because $\E\left[\left\Vert \bep_{t}^{u}\right\Vert ^{2}\right]=\E\left[\left\Vert \bxi_{t}\right\Vert ^{2}\right]$
could be $+\infty$ under Assumption \ref{assu:general-heavy-noise}.
\item In contrast, loosely speaking, setting $\tau_{t}<+\infty$ would ensure
that both $\E\left[\left\Vert \bep_{t}^{u}\right\Vert ^{2}\right]$
and $\E\left[\left\Vert \bep_{t}^{b}\right\Vert ^{2}\right]$ are
finite, whose upper bounds could be further controlled by picking
$\tau_{t}$ properly. Finally, Clipped SGD would converge under carefully
designed $\eta_{t}$ and $\tau_{t}$.
\end{itemize}
From the above comparison (which though is not strictly rigorous),
one can intuitively think that the key thing done by gradient clipping
is making the second moment of the error term $\E\left[\left\Vert \bep_{t}\right\Vert ^{2}\right]$
in (\ref{eq:clipped-sgd-2}) be bounded even when the noise $\bxi_{t}$
only has a finite $\p$-th moment.

\subsection{What does Gradient Normalization do?\label{subsec:normalization}}

By the discussion in the previous subsection, we can see that gradient
clipping is used to control the second moment of the error term. A
natural thought could be that we may avoid gradient clipping if the
error term $\left\Vert \bep_{t}\right\Vert ^{2}$ in (\ref{eq:clipped-sgd-2})
is in a lower order, for example, say $\left\Vert \bep_{t}\right\Vert ^{\p}$
which can be bounded in expectation directly without clipping. However,
because $\p$ is not necessarily known, we could aim to decrease the
order of $\left\Vert \bep_{t}\right\Vert $ from $2$ to $1$, i.e.,
the extreme case.

The above thought experiment may immediately help the reader who is
familiar with the optimization literature recall the gradient normalization
technique, in the analysis of which first-order terms always show
up. As such, it is reasonable to expect that Batched NSGDM can converge
under heavy-tailed noises even without knowing $\p$ due to gradient
normalization.

From now on, we start analyzing Algorithm \ref{alg:NSGDM} rigorously
and formally showing how gradient normalization overcomes heavy-tailed
noises.
\begin{lem}
\label{lem:main-descent-lemma}Under Assumptions \ref{assu:lb} and
\ref{assu:smooth} (with $L_{1}=0$), let $\Delta_{1}\defeq F(\bx_{1})-F_{*}$,
then Algorithm \ref{alg:NSGDM} guarantees
\begin{equation}
\sum_{t=1}^{T}\eta_{t}\E\left[\left\Vert \nabla F(\bx_{t})\right\Vert \right]\leq\Delta_{1}+\frac{L_{0}\sum_{t=1}^{T}\eta_{t}^{2}}{2}+\sum_{t=1}^{T}2\eta_{t}\E\left[\left\Vert \bep_{t}\right\Vert \right],\label{eq:main-descent-lemma}
\end{equation}
where $\bep_{t}\defeq\bm{_{t}}-\nabla F(\bx_{t}),\forall t\in\left[T\right]$.
\end{lem}

Lemma \ref{lem:main-descent-lemma} provides a basic inequality used
in the analysis of Batched NSGDM. We remark that this inequality is
not new and has been proved many times before, even when $L_{1}>0$
(e.g., see \citet{pmlr-v119-cutkosky20b,jin2021non}). But for completeness,
the proof is provided in Appendix \ref{sec:Full-Theorems}.

As mentioned earlier, the first moment $\E\left[\left\Vert \bep_{t}\right\Vert \right]$
is exactly what we want, which we hope could be finite even under
Assumption \ref{assu:general-heavy-noise}. To bound this term, we
first recall the following widely used decomposition when studying
gradient normalization \citep{pmlr-v119-cutkosky20b,cutkosky2021high,jin2021non,pmlr-v195-liu23c}.
\begin{lem}
\label{lem:main-decomposition}Algorithm \ref{alg:NSGDM} guarantees
\begin{equation}
\bep_{t}=\beta_{1:t}\bep_{0}+\sum_{s=1}^{t}\beta_{s:t}\bD_{s}+\sum_{s=1}^{t}(1-\beta_{s})\beta_{s+1:t}\bxi_{s},\forall t\in\left[T\right],\label{eq:main-decomposition}
\end{equation}
where 
\begin{eqnarray*}
\bep_{0}\triangleq\bg_{1}-\nabla F(\bx_{1}), & \bD_{t}\defeq\begin{cases}
\nabla F(\bx_{t-1})-\nabla F(\bx_{t}) & 2\leq t\leq T\\
\bzero & t=1
\end{cases}, & \bxi_{t}\defeq\bg_{t}-\nabla F(\bx_{t}).
\end{eqnarray*}
\end{lem}

\begin{proof}
By the definition of $\bep_{t}$ when $t\geq2$,
\begin{align*}
\bep_{t} & =\bm{_{t}}-\nabla F(\bx_{t})=\beta_{t}\bm{_{t-1}}+(1-\beta_{t})\bg_{t}-\nabla F(\bx_{t})\\
 & =\beta_{t}(\bm{_{t-1}}-\nabla F(\bx_{t-1}))+\beta_{t}(\nabla F(\bx_{t-1})-\nabla F(\bx_{t}))+(1-\beta_{t})(\bg_{t}-\nabla F(\bx_{t}))\\
 & =\beta_{t}\bep_{t-1}+\beta_{t}\bD_{t}+(1-\beta_{t})\bxi_{t}.
\end{align*}
One can verify that the above equation also holds when $t=1$ under
our notations (recall $\bm{_{0}}=\bg_{1}$ in Algorithm \ref{alg:NSGDM}).
Unrolling the equation recursively, we can finally obtain the desired
result.
\end{proof}

After plugging (\ref{eq:main-decomposition}) into (\ref{eq:main-descent-lemma}),
as one can imagine, our remaining task is to upper bound $\E\left[\left\Vert \bep_{0}\right\Vert \right]$,
$\E\left[\left\Vert \sum_{s=1}^{t}\beta_{s:t}\bD_{s}\right\Vert \right]$
and $\E\left[\left\Vert \sum_{s=1}^{t}(1-\beta_{s})\beta_{s+1:t}\bxi_{s}\right\Vert \right]$.
For simplicity, we consider the batch size $B=1$ in the following. 

First, note that $\E\left[\left\Vert \sum_{s=1}^{t}\beta_{s:t}\bD_{s}\right\Vert \right]\leq\sum_{s=1}^{t}\beta_{s:t}\E\left[\left\Vert \bD_{s}\right\Vert \right]$
and $\left\Vert \bD_{s}\right\Vert $ can be bounded by $L_{0}$-smoothness
easily (even under $(L_{0},L_{1})$-smoothness), hence we skip the
calculation here. Next, when $B=1$, one can use H\"{o}lder's inequality
to bound $\E\left[\left\Vert \bep_{0}\right\Vert \right]\leq\left(\E\left[\left\Vert \bep_{0}\right\Vert ^{\p}\right]\right)^{\frac{1}{\p}}\leq\sigma_{0}$
where the last step is by Assumption \ref{assu:general-heavy-noise}
when $\sigma_{1}=0$. For the left term $\E\left[\left\Vert \sum_{s=1}^{t}(1-\beta_{s})\beta_{s+1:t}\bxi_{s}\right\Vert \right]$:
\begin{itemize}
\item When $\p=2$, prior works like \citet{pmlr-v119-cutkosky20b} invoke
H\"{o}lder's inequality to have
\begin{align*}
\E\left[\left\Vert \sum_{s=1}^{t}(1-\beta_{s})\beta_{s+1:t}\bxi_{s}\right\Vert \right] & \leq\sqrt{\E\left[\left\Vert \sum_{s=1}^{t}(1-\beta_{s})\beta_{s+1:t}\bxi_{s}\right\Vert ^{2}\right]}\\
 & \leq\sqrt{\sum_{s=1}^{t}((1-\beta_{s})\beta_{s+1:t}\sigma_{0})^{2}},
\end{align*}
where the last step is due to $\E\left[\left\langle \bxi_{s},\bxi_{t}\right\rangle \right]=0$
for every cross term and $\E\left[\left\Vert \bxi_{s}\right\Vert ^{2}\right]\leq\sigma_{0}^{2}$.
\item Naturally, when noises only have finite $\p$-th moments, one may
want to apply H\"{o}lder's inequality in the following form,
\[
\E\left[\left\Vert \sum_{s=1}^{t}(1-\beta_{s})\beta_{s+1:t}\bxi_{s}\right\Vert \right]\leq\left(\E\left[\left\Vert \sum_{s=1}^{t}(1-\beta_{s})\beta_{s+1:t}\bxi_{s}\right\Vert ^{\p}\right]\right)^{\frac{1}{\p}}.
\]
However, a critical issue here is that $\left\Vert \cdot\right\Vert ^{\p}$
cannot be expanded like $\left\Vert \cdot\right\Vert ^{2}$ to make
the cross term $\left\langle \bxi_{s},\bxi_{t}\right\rangle $ show
up, which fails the analysis.
\end{itemize}
As such, how to establish a meaningful bound on $\E\left[\left\Vert \sum_{s=1}^{t}(1-\beta_{s})\beta_{s+1:t}\bxi_{s}\right\Vert \right]$
is the novel part of our proofs, which essentially differs from the
existing analysis when $\p=2$.

Fix $t\in\left[T\right]$, to ease the notation, we denote by $\bv_{s}\defeq(1-\beta_{s})\beta_{s+1:t}\bxi_{s},\forall s\in\left[t\right]$.
Hence, our goal can be summarized as bounding the norm of $\sum_{s=1}^{t}\bv_{s}$,
where $\bv_{s}$ is a vector-valued martingale difference sequence
(MDS). To do so, we introduce the following inequality, which is the
core of our analysis.
\begin{lem}
\label{lem:main-core}Given a sequence of integrable random vectors
$\bv_{t}\in\R^{d},\forall t\in\N$ such that $\E\left[\bv_{t}\mid\F_{t-1}\right]=\bzero$
where $\F_{t}\defeq\sigma\left(\bv_{1},\cdots,\bv_{t}\right)$ is
the natural filtration, then for any $\p\in\left[1,2\right]$, there
is
\begin{equation}
\E\left[\left\Vert \sum_{t=1}^{T}\bv_{t}\right\Vert \right]\le2\sqrt{2}\E\left[\left(\sum_{t=1}^{T}\left\Vert \bv_{t}\right\Vert ^{\p}\right)^{\frac{1}{\p}}\right],\forall T\in\N.\label{eq:main-core}
\end{equation}
\end{lem}

At first glance, (\ref{eq:main-core}) seems wrong because it provides
$\E\left[\left\Vert \sum_{t=1}^{T}\bv_{t}\right\Vert \right]\leq2\sqrt{2}\E\left[\sqrt{\sum_{t=1}^{T}\left\Vert \bv_{t}\right\Vert ^{2}}\right]$
by taking $\p=2$. However, one may only expect $\E\left[\left\Vert \sum_{t=1}^{T}\bv_{t}\right\Vert \right]\leq\sqrt{\sum_{t=1}^{T}\E\left[\left\Vert \bv_{t}\right\Vert ^{2}\right]}$
to hold. So why is (\ref{eq:main-core}) true? Intuitively, this is
because (\ref{eq:main-core}) is only stated for the MDS in contrast
to H\"{o}lder's inequality being able to apply to any sequence. 

To let the reader believe Lemma \ref{lem:main-core} is correct, we
first consider the case of $d=1$ and recall the famous Burkholder-Davis-Gundy
(BDG) inequality.
\begin{lem}
\label{lem:main-BDG}(Burkholder-Davis-Gundy Inequality \citep{burkholder1966martingale,burkholder1970extrapolation,davis1970intergrability},
simplified version) Given a discrete martingale $X_{t}\in\R$ with
$X_{0}=0$, then there exists a constant $C_{1}>0$ such that
\[
\E\left[\max_{t\in\left[T\right]}\left|X_{t}\right|\right]\leq C_{1}\E\left[\sqrt{\sum_{t=1}^{T}(X_{t}-X_{t-1})^{2}}\right],\forall T\in\N.
\]
\end{lem}

Let $X_{t}\defeq\sum_{s=1}^{t}\bv_{s}$, BDG inequality immediately
implies (\ref{eq:main-core}) under $d=1$ and $\p=2$ (up to a constant)
due to 
\begin{align*}
\E\left[\left|\sum_{t=1}^{T}\bv_{t}\right|\right] & =\E\left[\left|X_{T}\right|\right]\leq\E\left[\max_{t\in\left[T\right]}\left|X_{t}\right|\right]\leq C_{1}\E\left[\sqrt{\sum_{t=1}^{T}(X_{t}-X_{t-1})^{2}}\right]=C_{1}\E\left[\sqrt{\sum_{t=1}^{T}\left|\bv_{t}\right|^{2}}\right].
\end{align*}
One more step, by noticing $\left\Vert \cdot\right\Vert \leq\left\Vert \cdot\right\Vert _{\p}$
for $\p\in\left[1,2\right]$, Lemma \ref{lem:main-core} thereby holds
when $d=1$ by
\[
\E\left[\left|\sum_{t=1}^{T}\bv_{t}\right|\right]\leq C_{1}\E\left[\sqrt{\sum_{t=1}^{T}\left|\bv_{t}\right|^{2}}\right]\leq C_{1}\E\left[\left(\sum_{t=1}^{T}\left|\bv_{t}\right|^{\p}\right)^{\frac{1}{\p}}\right].
\]

Though we have applied BDG inequality to prove Lemma \ref{lem:main-core}
for $d=1$, extending the above analysis to the high-dimensional case
is not obvious and could be non-trivial. 

Here, inspired by \citet{pmlr-v65-rakhlin17a}, we will provide a
simple proof of Lemma \ref{lem:main-core} via the regret analysis
from \textit{online learning}. Specifically, we will prove the regret
bound of the famous AdaGrad algorithm \citep{McMahanS10,duchi2011adaptive}
implies Lemma \ref{lem:main-core}, which is kindly surprising (at
least in our opinion) and shows the impressive power of online learning.
Due to space limitations, the proof of Lemma \ref{lem:main-core}
is deferred to Appendix \ref{sec:core}.

Before moving on, we make two comments on Lemma \ref{lem:main-core}.
First, as one can see, Lemma \ref{lem:main-core} holds for any $\p\in\left[1,2\right]$
meaning that this analysis is automatically adaptive to the tail index
$\p$. Next, one may wonder why we keep the expectation outside on
the R.H.S. of (\ref{eq:main-core}) instead of putting it inside by
H\"{o}lder's inequality. This is because under the full version of
Assumption \ref{assu:general-heavy-noise} (i.e., $\sigma_{1}>0$),
making expectations inside may fail the analysis. For details, we
refer the reader to the proof of Lemma \ref{lem:err-bound} in the
appendix.

Now, we can apply Lemma \ref{lem:main-core} to bound $\E\left[\left\Vert \sum_{s=1}^{t}(1-\beta_{s})\beta_{s+1:t}\bxi_{s}\right\Vert \right]$
under Assumption \ref{assu:general-heavy-noise} with $\sigma_{1}=0$
and then obtain the following inequality for $\E\left[\left\Vert \bep_{t}\right\Vert \right]$
by combining the previous bounds on $\E\left[\left\Vert \bep_{0}\right\Vert \right]$
and $\E\left[\left\Vert \sum_{s=1}^{t}\beta_{s:t}\bD_{s}\right\Vert \right]$. 
\begin{lem}
\label{lem:main-err-bound}Under Assumptions \ref{assu:smooth} (with
$L_{1}=0$), \ref{assu:unbias} (with $B=1$) and \ref{assu:general-heavy-noise}
(with $\sigma_{1}=0$), then Algorithm \ref{alg:NSGDM} guarantees
\[
\E\left[\left\Vert \bep_{t}\right\Vert \right]\leq2\sqrt{2}\left[\beta_{1:t}\sigma_{0}+\left(\sum_{s=1}^{t}(1-\beta_{s})^{\p}(\beta_{s+1:t})^{\p}\right)^{\frac{1}{\p}}\sigma_{0}\right]+\sum_{s=2}^{t}\beta_{s:t}L_{0}\eta_{s-1},\forall t\in\left[T\right].
\]
\end{lem}

The full version of the above result, Lemma \ref{lem:err-bound},
that works for any $L_{1},B,\sigma_{1}$ can be found in Appendix
\ref{sec:Full-Theorems}, which requires extra effort as mentioned
earlier.

Finally, equipped with Lemmas \ref{lem:main-descent-lemma} and \ref{lem:main-err-bound},
we are able to prove the convergence of Batched NSGDM under the classical
smooth condition and heavy-tailed noises by plugging in the stepsize
and momentum parameter introduced in Theorems \ref{thm:main-fixed-full-info}
and \ref{thm:main-fixed-partial-info}, respectively.

Before ending this section, we briefly talk about the intuition behind
the rate $\O(T^{\frac{1-\p}{2\p}})$ achieved when the tail index
$\p$ is unknown, i.e., Theorem \ref{thm:main-fixed-partial-info}.
Actually, we take a quite simple strategy: setting the stepsize and
the momentum parameter while pretending the tail index $\p$ to be
$2$. Amazingly, this straightforward policy is already enough to
guarantee convergence even if there is no prior information on $\p$.

\section{Conclusion and Future Work\label{sec:conclusion}}

In this work, we present the first optimal expected convergence result
under heavy-tailed noises but without gradient clipping, which is
instead achieved by gradient normalization. More specifically, we
study the existing Batched NSGDM algorithm and prove it converges
in expectation at an optimal $\O(T^{\frac{1-\p}{3\p-2}})$ rate. Additionally,
the order of problem-dependent parameters in our upper bound is also
the first to be tight as indicated by a newly matched lower bound
improved from the prior work. One step further, we initiate the study
of convergence under heavy-tailed noises but without knowing the tail
index $\p$ and then obtain the first provable $\O(T^{\frac{1-\p}{2\p}})$
rate. Thus, our work suggests gradient normalization is a powerful
tool for dealing with heavy-tailed noises, which we believe will bring
new insights into the optimization community and open potential ways
for future algorithm design. 

However, there still remain some directions worth exploring, and we
list three specific topics here:

\textbf{Minimax rate for unknown tail index $\p$.} As previously
discussed, to achieve the minimax $\Theta(T^{\frac{1-\p}{3\p-2}})$
rate under heavy-tailed noises, all optimal algorithms so far require
knowing the tail index $\p$. Thus, it would be interesting to consider
the optimal upper/lower bound of the convergence rate when $\p$ is
unknown. We provide two concrete problems here and hope them being
addressed in the future: 1. When lacking any prior information on
$\p$, is it possible to find an algorithm that can improve our new
$\O(T^{\frac{1-\p}{2\p}})$ upper bound to the best-known $\O(T^{\frac{1-\p}{3\p-2}})$
rate? 2. If the answer to the former question is negative, what is
the corresponding lower bound when $\p$ is unknown?

\textbf{Adaptive gradient methods. }Though we have established the
first convergence result under heavy-tailed noises without gradient
clipping, the Batched NSGDM algorithm we studied is not commonly used
in practice. In comparison, the family of adaptive gradient methods
(e.g., AdaGrad \citep{McMahanS10,duchi2011adaptive}, RMSprop \citep{tieleman2012lecture},
Adam \citep{kingma2014adam} and their variants) is more popular and
has been widely implemented nowadays, especially when training neural
networks. Surprisingly, their performances are still good even though
the stochastic noises are empirically observed to be heavy-tailed.
However, as far as we know, no rigorous theoretical justification
has been established to show that adaptive gradient methods can converge
under heavy-tailed noises. Hence, it is worth studying and trying
to close this important gap between theory and practice. 

\textbf{Time-varying choices.} Another potential extension is to study
the time-varying stepsize and momentum parameter to make the algorithm
more practical, which we believe is possible given our general lemmas.

\clearpage

\section*{Acknowledgments}

This work is supported by the NSF grant ECCS-2419564. We also thank
the anonymous reviewers for their valuable comments.

\bibliographystyle{iclr2025_conference}
\bibliography{ref}

\clearpage

\appendix

\section{An Example Fails the Existing Heavy-Tailed Noises Assumption\label{sec:example}}

In this section, we provide a simple one-dimensional case that violates
the previously used finite $\p$-th moment assumption but satisfies
our new generalized heavy-tailed noises condition, Assumption \ref{assu:general-heavy-noise}.
Extending the example to high dimensions is straightforward, which
is left to the reader.
\begin{example}
\label{exa:noises}Let $F(\bx)\defeq\frac{1}{2}\E_{a,b}\left[(a\bx-b)^{2}\right]$
where $\bx\in\R$, $a\defeq\mathrm{Bernoulli}(q)$ for some $q\in\left(0,1\right)$,
and $b\defeq a\bx_{*}+\omega$ where $\bx_{*}\in\R$ is fixed and
$\omega$ is a centered random variable being independent of $a$
and only has a finite $\p$-th moment (i.e., $\E_{\omega}\left[\left|\omega\right|^{\p}\right]\leq\sigma^{\p}$
for some $\sigma\geq0$). For the stochastic gradient $\bg(\bx,a,b)\defeq a^{2}\bx-ab$
and the true gradient $\nabla F(\bx)\defeq\E_{a}\left[a^{2}\right]\bx-\E_{a,b}\left[ab\right]$,
we let the the noise be $\bxi(\bx)\defeq\bg(\bx,a,b)-\nabla F(\bx)$,
then
\begin{itemize}
\item Assumption \ref{assu:general-heavy-noise} is failed for any pair
$(\sigma_{0},0)$ where $\sigma_{0}\geq0$ can be arbitrary;
\item Assumption \ref{assu:general-heavy-noise} is satisfied for a certain
pair $(\sigma_{0},\sigma_{1})$ where $\sigma_{1}>0$.
\end{itemize}
\end{example}

\begin{proof}
Note that $\bg(\bx,a,b)$ and $\nabla F(\bx)$ can be simplified into
\[
\bg(\bx,a,b)=a^{2}(\bx-\bx_{*})-a\omega,\quad\nabla F(\bx)=q(\bx-\bx_{*}).
\]
Thus, we know
\begin{align*}
\bxi(\bx) & =(a^{2}-q)(\bx-\bx_{*})-a\omega\\
\Rightarrow\E_{a,b}\left[\left|\bxi(\bx)\right|^{\p}\right] & =\E_{a,\omega}\left[\left|(a^{2}-q)(\bx-\bx_{*})-a\omega\right|^{\p}\right]\\
 & =(1-q)q^{\p}\left|\bx-\bx_{*}\right|^{\p}+q\E_{\omega}\left[\left|(1-q)(\bx-\bx_{*})-\omega\right|^{\p}\right].
\end{align*}

On the one side, we have
\[
\E_{a,b}\left[\left|\bxi(\bx)\right|^{\p}\right]\geq(1-q)q^{\p}\left|\bx-\bx_{*}\right|^{\p}\overset{\bx\to\pm\infty}{\longrightarrow}+\infty,
\]
Therefore, $\E_{a,b}\left[\left|\bxi(\bx)\right|^{\p}\right]\leq\sigma_{0}^{\p}$
cannot hold for any $\sigma_{0}\geq0$.

On the other side, we observe
\begin{align*}
\E_{\omega}\left[\left|(1-q)(\bx-\bx_{*})-\omega\right|^{\p}\right] & \leq\E_{\omega}\left[\left|(1-q)\left|\bx-\bx_{*}\right|+\left|\omega\right|\right|^{\p}\right]\\
 & \leq2^{\p-1}(1-q)^{\p}\left|\bx-\bx_{*}\right|^{\p}+2^{\p-1}\E_{\omega}\left[\left|\omega\right|^{\p}\right]\\
 & \leq2^{\p-1}(1-q)^{\p}\left|\bx-\bx_{*}\right|^{\p}+2^{\p-1}\sigma^{\p}.
\end{align*}
Hence,
\begin{align*}
\E_{a,b}\left[\left|\bxi(\bx)\right|^{\p}\right] & \leq2^{\p-1}q\sigma^{\p}+\left[(1-q)q^{\p}+2^{\p-1}q(1-q)^{\p}\right]\left|\bx-\bx_{*}\right|^{\p}\\
 & =2^{\p-1}q\sigma^{\p}+\left[1-q+2^{\p-1}q^{1-\p}(1-q)^{\p}\right]\left|\nabla F(\bx)\right|^{\p}.
\end{align*}
So Assumption \ref{assu:general-heavy-noise} is satisfied for $\sigma_{0}^{\p}\defeq2^{\p-1}q\sigma^{\p}$
and $\sigma_{1}^{\p}\defeq1-q+2^{\p-1}q^{1-\p}(1-q)^{\p}$.
\end{proof}

\section{Further Discussion on $L_{1}$ in Theorem \ref{thm:main-fixed-partial-info}\label{sec:L1}}

One may think knowing $L_{1}$ is necessary under the present version
of $(L_{0},L_{1})$-smoothness (i.e., Assumption \ref{assu:smooth}).
It is because the current form of $(L_{0},L_{1})$-smoothness can
be only invoked under the hard constraint $\left\Vert \bx-\by\right\Vert \leq\frac{1}{L_{1}}$.
As such, regardless of whichever two points $\bx_{s}$ and $\bx_{t}$
we want to apply to the descent inequality in Lemma \ref{lem:descent-lemma}
(which serves as the cornerstone in the whole proof), they have to
satisfy $\left\Vert \bx_{s}-\bx_{t}\right\Vert \leq\frac{1}{L_{1}}$.
In other words, one has to know $L_{1}$ to set the stepsize to ensure
Lemma \ref{lem:descent-lemma} can be used. Moreover, even if when
$\p=2$ or under a weaker condition on noises (e.g., almost surely
bounded noises), prior works under exactly the same form of Assumption
\ref{assu:smooth} require the value of $L_{1}$ (e.g., see \citet{jin2021non,crawshaw2022robustness}),
even for the famous adaptive method -- AdaGrad \citep{pmlr-v195-faw23a,pmlr-v195-wang23a,hong2024revisiting}
and Adam \citep{wang2024convergence}.

However, this turns out to be too pessimistic. To see why, we first
claim that Assumption \ref{assu:smooth} is essentially equivalent
to the following condition (up to redefining the pair $(L_{0},L_{1})$)
\begin{equation}
\left\Vert \nabla F(\bx)-\nabla F(\by)\right\Vert \leq\left(L_{0}+L_{1}\left\Vert \nabla F(\bx)\right\Vert \right)\left\Vert \bx-\by\right\Vert \exp\left(L_{1}\left\Vert \bx-\by\right\Vert \right),\forall\bx,\by\in\R^{d}.\label{eq:equiv-ass}
\end{equation}

\begin{rem}
(\ref{eq:equiv-ass}) is also known to be equivalent to the $\mathcal{L_{\text{sym}}^{*}}(1)$
function class. See Definition 3 and Proposition 1.2 in \citet{pmlr-v202-chen23ar}
for details.
\end{rem}

This means that for any $\bx,\by\in\R^{d}$, we can still bound the
difference between their gradients even if the distance between them
is larger than $\frac{1}{L_{1}}$, but incurring an extra multiplicative
term as a penalty. As such, it is possible to combine the dynamic
stepsize and time-varying momentum, i.e., taking $\eta_{t}=\eta t^{-a}$
and $\beta_{t}=\beta t^{-b}$ for some $\eta,\beta,a,b>0$, to get
rid of knowing $L_{1}$ since we can control the size of gradient
norm during the warm-up period (i.e., the time before $\eta_{t}=\O(L_{1}^{-1})$)
and then follow a similar proof in our paper when $\eta_{t}=\O(L_{1}^{-1})$
to obtain the desired result. Whereas, as a trade-off, one has to
incur an exponential dependence on $L_{1}$ in the final convergence
rate as pointed out by \citet{pmlr-v238-hubler24a}.

Lastly, let us prove this important equivalence.
\begin{itemize}
\item (\ref{eq:equiv-ass}) implies Assumption \ref{assu:smooth}. This
is the easy case. Suppose (\ref{eq:equiv-ass}) holds, then for any
$\bx,\by\in\R^{d}$ satisfying $\left\Vert \bx-\by\right\Vert \leq\frac{1}{L_{1}}$,
there is
\[
\left\Vert \nabla F(\bx)-\nabla F(\by)\right\Vert \leq\left(eL_{0}+eL_{1}\left\Vert \nabla F(\bx)\right\Vert \right)\left\Vert \bx-\by\right\Vert ,
\]
which implies $F$ is $(eL_{0},eL_{1})$-smooth in the sense of Assumption
\ref{assu:smooth}.
\item Assumption \ref{assu:smooth} implies (\ref{eq:equiv-ass})\footnote{We especially thank the anonymous Reviewer tQr8 for suggesting this
direction during the discussion.}. This direction is not obvious, whose proof in the following is inspired
by \citet{pmlr-v202-chen23ar}. Given $\bx,\by\in\R^{d}$, we define
the linear interpolation $\bz_{\bx,\by}(t)\defeq t\bx+(1-t)\by,\forall t\in\left[0,1\right]$.
Note that for any $n\in\N$ satisfying $n\geq L_{1}\left\Vert \bx-\by\right\Vert $
and $i\in\left[n\right]$, there is
\[
\left\Vert \bz_{\bx,\by}(\nicefrac{i}{n})-\bz_{\bx,\by}(\nicefrac{i-1}{n})\right\Vert =\frac{\left\Vert \bx-\by\right\Vert }{n}\leq\frac{1}{L_{1}}.
\]
Therefore, we can use Assumption \ref{assu:smooth} to control $\left\Vert \nabla F(\bz_{\bx,\by}(\nicefrac{i}{n}))-\nabla F(\bz_{\bx,\by}(\nicefrac{i-1}{n}))\right\Vert $
and obtain the following bound whenever $n\geq L_{1}\left\Vert \bx-\by\right\Vert $,
\begin{align}
 & \left\Vert \nabla F(\bx)-\nabla F(\by)\right\Vert \nonumber \\
= & \left\Vert \sum_{i=1}^{n}\nabla F(\bz_{\bx,\by}(\nicefrac{i}{n}))-\nabla F(\bz_{\bx,\by}(\nicefrac{i-1}{n}))\right\Vert \leq\sum_{i=1}^{n}\left\Vert \nabla F(\bz_{\bx,\by}(\nicefrac{i}{n}))-\nabla F(\bz_{\bx,\by}(\nicefrac{i-1}{n}))\right\Vert \nonumber \\
\leq & \frac{\left\Vert \bx-\by\right\Vert }{n}\sum_{i=1}^{n}L_{0}+L_{1}\left\Vert \nabla F(\bz_{\bx,\by}(\nicefrac{i-1}{n}))\right\Vert =\frac{\left\Vert \bx-\by\right\Vert }{n}\sum_{i=1}^{n}h(\nicefrac{i-1}{n}),\label{eq:equive-ass-discrete}
\end{align}
where $h(t)\defeq L_{0}+L_{1}\left\Vert \nabla F(\bz_{\bx,\by}(t))\right\Vert ,\forall t\in\left[0,1\right]$.
We notice that $h$ is continuous because for any given $t\in\left(0,1\right)$
and $\varepsilon>0$, for $s\in\left[0,1\right]$ satisfying $\left|s-t\right|\leq\frac{\min\left\{ h(t),\varepsilon\right\} }{L_{1}h(t)\left\Vert \bx-\by\right\Vert }$,
there is
\begin{align*}
\left|h(s)-h(t)\right| & =L_{1}\left|\left\Vert \nabla F(\bz_{\bx,\by}(s))\right\Vert -\left\Vert \nabla F(\bz_{\bx,\by}(t))\right\Vert \right|\leq L_{1}\left\Vert \nabla F(\bz_{\bx,\by}(s))-\nabla F(\bz_{\bx,\by}(t))\right\Vert \\
 & \overset{(a)}{\leq}L_{1}h(t)\left\Vert \bz_{\bx,\by}(s)-\bz_{\bx,\by}(t)\right\Vert =L_{1}h(t)\left\Vert \bx-\by\right\Vert \left|s-t\right|\overset{(b)}{\leq}\varepsilon,
\end{align*}
where $(a)$ is due to $\left\Vert \bz_{\bx,\by}(s)-\bz_{\bx,\by}(t)\right\Vert =\left\Vert \bx-\by\right\Vert \left|s-t\right|\leq\frac{1}{L_{1}}$
and Assumption \ref{assu:smooth}, and $(b)$ holds by $\left|s-t\right|\leq\frac{\varepsilon}{L_{1}h(t)\left\Vert \bx-\by\right\Vert }.$
Thus, $h$ is integrable on $\left[0,1\right]$, which implies
\begin{equation}
\left\Vert \nabla F(\bx)-\nabla F(\by)\right\Vert \overset{\eqref{eq:equive-ass-discrete}}{\leq}\left\Vert \bx-\by\right\Vert \lim_{n\to\infty}\frac{1}{n}\sum_{i=1}^{n}h(\nicefrac{i-1}{n})=\left\Vert \bx-\by\right\Vert \int_{0}^{1}h(t)\d t.\label{eq:equiv-ass-int}
\end{equation}
By Lemma 2 and Proposition 1.2 in \citet{pmlr-v202-chen23ar}, (\ref{eq:equiv-ass-int})
is equivalent to (\ref{eq:equiv-ass}).
\end{itemize}

\section{Proof of the Core Lemma \ref{lem:main-core} and Some Discussions\label{sec:core}}

Before proving Lemma \ref{lem:main-core}, we need the following technique
tool, which is based on the regret analysis of the AdaGrad algorithm
as mentioned in Subsection \ref{subsec:normalization}.
\begin{lem}
\label{lem:OLO}(Based on Lemma 2 in \citet{pmlr-v65-rakhlin17a})
Given a sequence of vectors $\bv_{t}\in\R^{d},\forall t\in\N$, then
there exists a sequence of vectors $\bw_{t}\in\R^{d}$ such that $\left\Vert \bw_{t}\right\Vert \leq1$
and every $\bw_{t}$ only depends on $\bv_{1}$ to $\bv_{t-1}$ satisfying
\[
\left\Vert \sum_{t=1}^{T}\bv_{t}\right\Vert \le2\sqrt{2\sum_{t=1}^{T}\left\Vert \bv_{t}\right\Vert ^{2}}-\sum_{t=1}^{T}\left\langle \bv_{t},\bw_{t}\right\rangle ,\forall T\in\N.
\]
\end{lem}

\begin{proof}
W.l.o.g., we assume $\left\Vert \bv_{1}\right\Vert >0$. Otherwise,
we let $\tau\defeq\argmin\left\{ t\in\N:\left\Vert \bv_{t}\right\Vert >0\right\} $,
set $\bw_{t}\defeq\mathbf{0}\in\R^{d},\forall t\in\left[\tau-1\right]$,
and start the following proof at time $\tau$.

Let $\bw_{1}\defeq\mathbf{0}$ and recursively define $\bw_{t+1}\defeq\Pi_{\B^{d}}\left(\bw_{t}-\gamma_{t}\bv_{t}\right)$
where $\Pi_{\B^{d}}$ denotes the Euclidean projection operator onto
the unit ball $\B^{d}$ in $\R^{d}$ and $\gamma_{t}\defeq\sqrt{\frac{2}{\sum_{s=1}^{t}\left\Vert \bv_{s}\right\Vert ^{2}}}$.
Clearly, $\left\Vert \bw_{t}\right\Vert \leq1$ and $\bw_{t}$ only
depends on $\bv_{1}$ to $\bv_{t-1}$ by construction. Next, observe
that
\[
\bw_{t+1}=\Pi_{\B^{d}}\left(\bw_{t}-\gamma_{t}\bv_{t}\right)=\argmin_{\bw\in\B^{d}}\left\Vert \bw-\bw_{t}+\gamma_{t}\bv_{t}\right\Vert ^{2},
\]
which by the optimality condition implies for any $\bu\in\B^{d}$,
\begin{align*}
\left\langle \bw_{t+1}-\bw_{t}+\gamma_{t}\bv_{t},\bw_{t+1}-\bu\right\rangle  & \leq0\\
\Rightarrow\left\langle \bv_{t},\bw_{t+1}-\bu\right\rangle  & \leq\frac{\left\langle \bw_{t}-\bw_{t+1},\bw_{t+1}-\bu\right\rangle }{\gamma_{t}}\\
 & =\frac{\left\Vert \bu-\bw_{t}\right\Vert ^{2}-\left\Vert \bu-\bw_{t+1}\right\Vert ^{2}-\left\Vert \bw_{t}-\bw_{t+1}\right\Vert ^{2}}{2\gamma_{t}}\\
\Rightarrow\left\langle \bv_{t},\bw_{t}-\bu\right\rangle  & \leq\frac{\left\Vert \bu-\bw_{t}\right\Vert ^{2}-\left\Vert \bu-\bw_{t+1}\right\Vert ^{2}-\left\Vert \bw_{t}-\bw_{t+1}\right\Vert ^{2}}{2\gamma_{t}}+\left\langle \bv_{t},\bw_{t}-\bw_{t+1}\right\rangle \\
 & \overset{(a)}{\leq}\frac{\left\Vert \bu-\bw_{t}\right\Vert ^{2}-\left\Vert \bu-\bw_{t+1}\right\Vert ^{2}}{2\gamma_{t}}+\frac{\gamma_{t}\left\Vert \bv_{t}\right\Vert ^{2}}{2},
\end{align*}
where $(a)$ is by the Arithmetic Mean-Geometric Mean inequality.
Hence, for any $\bu\in\B^{d}$ and $T\in\N$,
\begin{align}
\sum_{t=1}^{T}\left\langle \bv_{t},\bw_{t}-\bu\right\rangle  & \leq\sum_{t=1}^{T}\frac{\left\Vert \bu-\bw_{t}\right\Vert ^{2}-\left\Vert \bu-\bw_{t+1}\right\Vert ^{2}}{2\gamma_{t}}+\frac{\gamma_{t}\left\Vert \bv_{t}\right\Vert ^{2}}{2}\nonumber \\
 & =\frac{\left\Vert \bu-\bw_{1}\right\Vert ^{2}}{2\gamma_{1}}-\frac{\left\Vert \bu-\bw_{T+1}\right\Vert ^{2}}{2\gamma_{T}}+\sum_{t=2}^{T}\frac{\left\Vert \bu-\bw_{t}\right\Vert ^{2}}{2}\left(\frac{1}{\gamma_{t}}-\frac{1}{\gamma_{t-1}}\right)+\sum_{t=1}^{T}\frac{\gamma_{t}\left\Vert \bv_{t}\right\Vert ^{2}}{2}\nonumber \\
 & \overset{(b)}{\leq}\frac{2}{\gamma_{1}}+\sum_{t=2}^{T}2\left(\frac{1}{\gamma_{t}}-\frac{1}{\gamma_{t-1}}\right)+\sum_{t=1}^{T}\frac{\gamma_{t}\left\Vert \bv_{t}\right\Vert ^{2}}{2}=\frac{2}{\gamma_{T}}+\sum_{t=1}^{T}\frac{\gamma_{t}\left\Vert \bv_{t}\right\Vert ^{2}}{2},\label{eq:OLO-1}
\end{align}
where $(b)$ is due to $\frac{1}{\gamma_{t}}-\frac{1}{\gamma_{t-1}}\geq0,\forall t\geq2$
by the definition of $\gamma_{t}=\sqrt{\frac{2}{\sum_{s=1}^{t}\left\Vert \bv_{s}\right\Vert ^{2}}}$
and $\left\Vert \bu-\bw_{t}\right\Vert \leq2$ since $\bu$ and $\bw_{t}$
are both in $\B^{d}$. In addition, let $\frac{1}{\gamma_{0}}\defeq0$,
we notice that
\begin{equation}
\sum_{t=1}^{T}\frac{\gamma_{t}\left\Vert \bv_{t}\right\Vert ^{2}}{2}=\sum_{t=1}^{T}\gamma_{t}\left(\frac{1}{\gamma_{t}^{2}}-\frac{1}{\gamma_{t-1}^{2}}\right)\leq\sum_{t=1}^{T}2\left(\frac{1}{\gamma_{t}}-\frac{1}{\gamma_{t-1}}\right)=\frac{2}{\gamma_{T}}.\label{eq:OLO-2}
\end{equation}

Combining (\ref{eq:OLO-1}) and (\ref{eq:OLO-2}) and using $\gamma_{T}=\sqrt{\frac{2}{\sum_{t=1}^{T}\left\Vert \bv_{t}\right\Vert ^{2}}}$,
we have
\[
\left\langle \sum_{t=1}^{T}\bv_{t},-\bu\right\rangle \leq2\sqrt{2\sum_{t=1}^{T}\left\Vert \bv_{t}\right\Vert ^{2}}-\sum_{t=1}^{T}\left\langle \bv_{t},\bw_{t}\right\rangle ,\forall\bu\in\B^{d},T\in\N.
\]
Finally, we use $\max_{\bu\in\B^{d}}\left\langle \sum_{t=1}^{T}\bv_{t},-\bu\right\rangle =\left\Vert \sum_{t=1}^{T}\bv_{t}\right\Vert $
to obtain
\[
\left\Vert \sum_{t=1}^{T}\bv_{t}\right\Vert \le2\sqrt{2\sum_{t=1}^{T}\left\Vert \bv_{t}\right\Vert ^{2}}-\sum_{t=1}^{T}\left\langle \bv_{t},\bw_{t}\right\rangle ,\forall T\in\N.
\]
\end{proof}

Now we are ready to prove Lemma \ref{lem:main-core}.

\begin{proof}[Proof of Lemma \ref{lem:main-core}]
By Lemma \ref{lem:OLO}, there exists a sequence of random vectors
$\bw_{t}\in\R^{d}$ such that $\bw_{t}\in\F_{t-1}$ satisfying 
\[
\left\Vert \sum_{t=1}^{T}\bv_{t}\right\Vert \le2\sqrt{2\sum_{t=1}^{T}\left\Vert \bv_{t}\right\Vert ^{2}}-\sum_{t=1}^{T}\left\langle \bv_{t},\bw_{t}\right\rangle ,\forall T\in\N.
\]
We take expectations on both sides to get
\[
\E\left[\left\Vert \sum_{t=1}^{T}\bv_{t}\right\Vert \right]\leq2\sqrt{2}\E\left[\sqrt{\sum_{t=1}^{T}\left\Vert \bv_{t}\right\Vert ^{2}}\right]-\sum_{t=1}^{T}\E\left[\left\langle \bv_{t},\bw_{t}\right\rangle \right].
\]
By noticing that 
\[
\E\left[\left\langle \bv_{t},\bw_{t}\right\rangle \right]=\E\left[\E\left[\left\langle \bv_{t},\bw_{t}\right\rangle \mid\F_{t-1}\right]\right]=\E\left[\left\langle \E\left[\bv_{t}\mid\F_{t-1}\right],\bw_{t}\right\rangle \right]=0,
\]
we find
\[
\E\left[\left\Vert \sum_{t=1}^{T}\bv_{t}\right\Vert \right]\leq2\sqrt{2}\E\left[\sqrt{\sum_{t=1}^{T}\left\Vert \bv_{t}\right\Vert ^{2}}\right].
\]
Finally, by observing $\sqrt{\sum_{t=1}^{T}\left\Vert \bv_{t}\right\Vert ^{2}}\leq\left(\sum_{t=1}^{T}\left\Vert \bv_{t}\right\Vert ^{\p}\right)^{\frac{1}{\p}},\forall\p\in\left[1,2\right]$,
the proof is hence complete.
\end{proof}

Lastly, we briefly talk about the difference between our Lemma \ref{lem:main-core}
and Lemma 4 in \citet{NEURIPS2023_ca24eb48}, the latter of which
essentially states an inequality $\E\left[\left\Vert \sum_{t=1}^{T}\bv_{t}\right\Vert ^{\p}\right]\leq2T\sigma^{\p}$
under the same setting as in our Lemma \ref{lem:main-core} and an
additional condition $\E\left[\left\Vert \bv_{t}\right\Vert ^{\p}\right]\leq\sigma,\forall t\in\left[T\right]$.
Though it is possible to employ their inequality to prove convergence
under the traditional heavy-tailed noises assumption (i.e., $\sigma_{1}=0$
in Assumption \ref{assu:general-heavy-noise}), whether it is applicable
under our new Assumption \ref{assu:general-heavy-noise} remains unclear
since dealing with the new term in the generalized heavy-tailed noises
assumption requires a finer conditional expectation argument as shown
in the proof of Lemma \ref{lem:err-bound}. Hence, we believe that
our new Lemma \ref{lem:main-core} is more general.

\section{Full Theorems and Other Missing Proofs\label{sec:Full-Theorems}}

\subsection{Upper Bounds}
\begin{thm}
\label{thm:fixed-full-info}(Full version of Theorem \ref{thm:main-fixed-full-info})
Under Assumptions \ref{assu:lb}, \ref{assu:smooth}, \ref{assu:unbias}
and \ref{assu:general-heavy-noise}, let $\Delta_{1}\defeq F(\bx_{1})-F_{*}$,
then for any $T\in\N$, by taking
\begin{align*}
\beta_{t} & \equiv\beta=1-\min\left\{ 1,\max\left\{ \left(\frac{\Delta_{1}L_{1}B^{\frac{\p-1}{\p}}+\sigma_{0}+\sigma_{1}\left\Vert \nabla F(\bx_{1})\right\Vert }{\sigma_{0}T}\right)^{\frac{\p}{2\p-1}},\left(\frac{\Delta_{1}L_{0}B^{\frac{2(\p-1)}{\p}}}{\sigma_{0}^{2}T}\right)^{\frac{\p}{3\p-2}}\right\} \right\} ,\\
\eta_{t} & \equiv\eta=\min\left\{ \sqrt{\frac{(1-\beta)\Delta_{1}}{L_{0}T}},\frac{1-\beta}{8L_{1}}\right\} ,\quad B=\max\left\{ \left\lceil (16\sqrt{2}\sigma_{1})^{\frac{\p}{\p-1}}\right\rceil ,1\right\} ,
\end{align*}
Algorithm \ref{alg:NSGDM} guarantees
\begin{align*}
\frac{1}{T}\sum_{t=1}^{T}\E\left[\left\Vert \nabla F(\bx_{t})\right\Vert \right]= & \O\left(\frac{\Delta_{1}L_{1}+(\sigma_{0}+\sigma_{1}\left\Vert \nabla F(\bx_{1})\right\Vert )/B^{\frac{\p-1}{\p}}}{T}+\sqrt{\frac{\Delta_{1}L_{0}}{T}}\right.\\
 & \left.+\frac{(\Delta_{1}L_{1}+(\sigma_{0}+\sigma_{1}\left\Vert \nabla F(\bx_{1})\right\Vert )/B^{\frac{\p-1}{\p}})^{\frac{\p-1}{2\p-1}}\sigma_{0}^{\frac{\p}{2\p-1}}}{(BT)^{\frac{\p-1}{2\p-1}}}+\frac{(\Delta_{1}L_{0})^{\frac{\p-1}{3\p-2}}\sigma_{0}^{\frac{\p}{3\p-2}}}{(BT)^{\frac{\p-1}{3\p-2}}}\right).
\end{align*}
\end{thm}

\begin{proof}
Because $\eta_{t}\equiv\eta\leq\frac{1-\beta}{8L_{1}}\leq\frac{1}{L_{1}},\forall t\in\left[T\right]$,
we can safely invoke Lemma \ref{lem:descent-lemma} with $\eta_{t}\equiv\eta,\forall t\in\left[T\right]$
to obtain
\begin{equation}
\sum_{t=1}^{T}\eta\E\left[\left\Vert \nabla F(\bx_{t})\right\Vert \right]\leq\Delta_{1}+\frac{\eta^{2}L_{0}T}{2}+\sum_{t=1}^{T}2\eta\E\left[\left\Vert \bep_{t}\right\Vert \right]+\sum_{t=1}^{T}\frac{\eta^{2}L_{1}}{2}\E\left[\left\Vert \nabla F(\bx_{t})\right\Vert \right].\label{eq:fixed-full-info-1}
\end{equation}

Next, by Lemma \ref{lem:err-bound} with $\eta_{t}\equiv\eta,\beta_{t}\equiv\beta,\forall t\in\left[T\right]$,
we have for any $t\in\left[T\right]$,
\begin{align*}
\E\left[\left\Vert \bep_{t}\right\Vert \right]\leq & \frac{2\sqrt{2}}{B^{\frac{\p-1}{\p}}}\left[\beta^{t}\left(\sigma_{0}+\sigma_{1}\left\Vert \nabla F(\bx_{1})\right\Vert \right)+(1-\beta)\left(\sum_{s=1}^{t}\beta^{\p(t-s)}\right)^{\frac{1}{\p}}\sigma_{0}\right]\\
 & +\sum_{s=2}^{t}\beta^{t-s+1}\left(L_{0}+L_{1}\E\left[\left\Vert \nabla F(\bx_{s-1})\right\Vert \right]\right)\eta+\frac{2\sqrt{2}}{B^{\frac{\p-1}{\p}}}\sum_{s=1}^{t}(1-\beta)\beta^{t-s}\sigma_{1}\E\left[\left\Vert \nabla F(\bx_{s})\right\Vert \right]\\
\leq & \frac{2\sqrt{2}}{B^{\frac{\p-1}{\p}}}\left[\beta^{t}\left(\sigma_{0}+\sigma_{1}\left\Vert \nabla F(\bx_{1})\right\Vert \right)+\frac{1-\beta}{(1-\beta^{\p})^{\frac{1}{\p}}}\sigma_{0}\right]+\frac{\beta\eta L_{0}}{1-\beta}\\
 & +\sum_{s=2}^{t}\beta^{t-s+1}\eta L_{1}\E\left[\left\Vert \nabla F(\bx_{s-1})\right\Vert \right]+\frac{2\sqrt{2}}{B^{\frac{\p-1}{\p}}}\sum_{s=1}^{t}(1-\beta)\beta^{t-s}\sigma_{1}\E\left[\left\Vert \nabla F(\bx_{s})\right\Vert \right]\\
\leq & \frac{2\sqrt{2}}{B^{\frac{\p-1}{\p}}}\left[\beta^{t}\left(\sigma_{0}+\sigma_{1}\left\Vert \nabla F(\bx_{1})\right\Vert \right)+(1-\beta)^{\frac{\p-1}{\p}}\sigma_{0}\right]+\frac{\beta\eta L_{0}}{1-\beta}\\
 & +\sum_{s=2}^{t}\beta^{t-s+1}\eta L_{1}\E\left[\left\Vert \nabla F(\bx_{s-1})\right\Vert \right]+\frac{2\sqrt{2}}{B^{\frac{\p-1}{\p}}}\sum_{s=1}^{t}(1-\beta)\beta^{t-s}\sigma_{1}\E\left[\left\Vert \nabla F(\bx_{s})\right\Vert \right],
\end{align*}
where we use $\beta^{\p}\leq\beta$ when $\p\geq1$ and $\beta\leq1$
in the last step. As such, we know 
\begin{align}
\sum_{t=1}^{T}2\eta\E\left[\left\Vert \bep_{t}\right\Vert \right]\leq & \frac{4\sqrt{2}\eta}{B^{\frac{\p-1}{\p}}}\sum_{t=1}^{T}\left[\beta^{t}\left(\sigma_{0}+\sigma_{1}\left\Vert \nabla F(\bx_{1})\right\Vert \right)+(1-\beta)^{\frac{\p-1}{\p}}\sigma_{0}\right]+\sum_{t=1}^{T}\frac{2\beta\eta^{2}L_{0}}{1-\beta}\nonumber \\
 & +\sum_{t=1}^{T}\sum_{s=2}^{t}2\beta^{t-s+1}\eta^{2}L_{1}\E\left[\left\Vert \nabla F(\bx_{s-1})\right\Vert \right]+\frac{4\sqrt{2}\eta}{B^{\frac{\p-1}{\p}}}\sum_{t=1}^{T}\sum_{s=1}^{t}(1-\beta)\beta^{t-s}\sigma_{1}\E\left[\left\Vert \nabla F(\bx_{s})\right\Vert \right]\nonumber \\
\leq & \frac{4\sqrt{2}\eta}{B^{\frac{\p-1}{\p}}}\left[\frac{\beta\left(\sigma_{0}+\sigma_{1}\left\Vert \nabla F(\bx_{1})\right\Vert \right)}{1-\beta}+(1-\beta)^{\frac{\p-1}{\p}}T\sigma_{0}\right]+\frac{2\beta\eta^{2}L_{0}T}{1-\beta}\nonumber \\
 & +\sum_{t=1}^{T}\left(\frac{2\beta\eta^{2}L_{1}}{1-\beta}+\frac{4\sqrt{2}\eta\sigma_{1}}{B^{\frac{\p-1}{\p}}}\right)\E\left[\left\Vert \nabla F(\bx_{t})\right\Vert \right],\label{eq:fixed-full-info-2}
\end{align}
where in the last step we use $\sum_{t=1}^{T}\sum_{s=i}^{t}\cdot=\sum_{s=i}^{T}\sum_{t=s}^{T}\cdot\leq\sum_{s=i}^{T}\sum_{t=s}^{\infty}\cdot$
when the summands are non-negative for $i\in\left\{ 1,2\right\} $.

Now plugging (\ref{eq:fixed-full-info-2}) into (\ref{eq:fixed-full-info-1})
to get
\begin{align}
 & \sum_{t=1}^{T}\eta\E\left[\left\Vert \nabla F(\bx_{t})\right\Vert \right]\nonumber \\
\leq & \Delta_{1}+\frac{(1+3\beta)\eta^{2}L_{0}T}{2(1-\beta)}+\frac{4\sqrt{2}\eta}{B^{\frac{\p-1}{\p}}}\left[\frac{\beta\left(\sigma_{0}+\sigma_{1}\left\Vert \nabla F(\bx_{1})\right\Vert \right)}{1-\beta}+(1-\beta)^{\frac{\p-1}{\p}}T\sigma_{0}\right]\nonumber \\
 & +\sum_{t=1}^{T}\left(\frac{(1+3\beta)\eta^{2}L_{1}}{2(1-\beta)}+\frac{4\sqrt{2}\eta\sigma_{1}}{B^{\frac{\p-1}{\p}}}\right)\E\left[\left\Vert \nabla F(\bx_{t})\right\Vert \right]\nonumber \\
\overset{(a)}{\leq} & \Delta_{1}+\frac{2\eta^{2}L_{0}T}{1-\beta}+\frac{4\sqrt{2}\eta}{B^{\frac{\p-1}{\p}}}\left[\frac{\sigma_{0}+\sigma_{1}\left\Vert \nabla F(\bx_{1})\right\Vert }{1-\beta}+(1-\beta)^{\frac{\p-1}{\p}}T\sigma_{0}\right]\nonumber \\
 & +\sum_{t=1}^{T}\left(\frac{2\eta^{2}L_{1}}{1-\beta}+\frac{4\sqrt{2}\eta\sigma_{1}}{B^{\frac{\p-1}{\p}}}\right)\E\left[\left\Vert \nabla F(\bx_{t})\right\Vert \right]\nonumber \\
\overset{(b)}{\leq} & \Delta_{1}+\frac{2\eta^{2}L_{0}T}{1-\beta}+\frac{4\sqrt{2}\eta}{B^{\frac{\p-1}{\p}}}\left[\frac{\sigma_{0}+\sigma_{1}\left\Vert \nabla F(\bx_{1})\right\Vert }{1-\beta}+(1-\beta)^{\frac{\p-1}{\p}}T\sigma_{0}\right]+\sum_{t=1}^{T}\frac{\eta}{2}\E\left[\left\Vert \nabla F(\bx_{t})\right\Vert \right],\label{eq:fixed-full-info-3}
\end{align}
where we use $\beta\leq1$ in $(a)$ and $\eta\leq\frac{1-\beta}{8L_{1}},B\geq(16\sqrt{2}\sigma_{1})^{\frac{\p}{\p-1}}$
in $(b)$. We observe that (\ref{eq:fixed-full-info-3}) implies
\begin{align}
 & \sum_{t=1}^{T}\E\left[\left\Vert \nabla F(\bx_{t})\right\Vert \right]\nonumber \\
\leq & \frac{2\Delta_{1}}{\eta}+\frac{4\eta L_{0}T}{1-\beta}+\frac{8\sqrt{2}}{B^{\frac{\p-1}{\p}}}\left[\frac{\sigma_{0}+\sigma_{1}\left\Vert \nabla F(\bx_{1})\right\Vert }{1-\beta}+(1-\beta)^{\frac{\p-1}{\p}}T\sigma_{0}\right]\label{eq:fixed-full-info-4}\\
\overset{(c)}{=} & \O\left(\frac{\Delta_{1}L_{1}+(\sigma_{0}+\sigma_{1}\left\Vert \nabla F(\bx_{1})\right\Vert )/B^{\frac{\p-1}{\p}}}{1-\beta}+\sqrt{\frac{\Delta L_{0}T}{1-\beta}}+\frac{(1-\beta)^{\frac{\p-1}{\p}}\sigma_{0}T}{B^{\frac{\p-1}{\p}}}\right)\nonumber \\
\overset{(d)}{=} & \O\left(\Delta_{1}L_{1}+(\sigma_{0}+\sigma_{1}\left\Vert \nabla F(\bx_{1})\right\Vert )/B^{\frac{\p-1}{\p}}+\sqrt{\Delta L_{0}T}\right.\nonumber \\
 & \left.+\frac{(\Delta_{1}L_{1}+(\sigma_{0}+\sigma_{1}\left\Vert \nabla F(\bx_{1})\right\Vert )/B^{\frac{\p-1}{\p}})^{\frac{\p-1}{2\p-1}}(\sigma_{0}T)^{\frac{\p}{2\p-1}}}{B^{\frac{\p-1}{2\p-1}}}+\frac{(\Delta_{1}L_{0})^{\frac{\p-1}{3\p-2}}\sigma_{0}^{\frac{\p}{3\p-2}}T^{\frac{2\p-1}{3\p-2}}}{B^{\frac{\p-1}{3\p-2}}}\right),\label{eq:fixed-full-info-5}
\end{align}
where we plug in $\eta=\min\left\{ \sqrt{\frac{(1-\beta)\Delta_{1}}{L_{0}T}},\frac{1-\beta}{8L_{1}}\right\} $
in $(c)$ and use the following estimation in $(d)$: Let $U\defeq\left(\frac{\Delta_{1}L_{1}B^{\frac{\p-1}{\p}}+\sigma_{0}+\sigma_{1}\left\Vert \nabla F(\bx_{1})\right\Vert }{\sigma_{0}T}\right)^{\frac{\p}{2\p-1}}$
and $V\defeq\left(\frac{\Delta_{1}L_{0}B^{\frac{2(\p-1)}{\p}}}{\sigma_{0}^{2}T}\right)^{\frac{\p}{3\p-2}}$,
we know $1-\beta=\min\left\{ 1,\max\left\{ U,V\right\} \right\} $.
Then there is
\begin{align*}
 & \frac{C_{1}}{1-\beta}+\sqrt{\frac{C_{2}}{1-\beta}}+C_{3}(1-\beta)^{\frac{\p-1}{\p}}\\
\le & C_{1}\left(1+\frac{1}{\max\left\{ U,V\right\} }\right)+\sqrt{C_{2}\left(1+\frac{1}{\max\left\{ U,V\right\} }\right)}+C_{3}\left(\max\left\{ U,V\right\} \right)^{\frac{\p-1}{\p}}\\
\le & C_{1}\left(1+\frac{1}{U}\right)+\sqrt{C_{2}\left(1+\frac{1}{V}\right)}+C_{3}\left(U^{\frac{\p-1}{\p}}+V^{\frac{\p-1}{\p}}\right)\\
\le & C_{1}+\sqrt{C_{2}}+\frac{C_{1}}{U}+\sqrt{\frac{C_{2}}{V}}+C_{3}\left(U^{\frac{\p-1}{\p}}+V^{\frac{\p-1}{\p}}\right),
\end{align*}
where
\begin{eqnarray*}
C_{1}\defeq\Delta_{1}L_{1}+(\sigma_{0}+\sigma_{1}\left\Vert \nabla F(\bx_{1})\right\Vert )/B^{\frac{\p-1}{\p}}, & C_{2}\defeq\Delta L_{0}T, & C_{3}\defeq\frac{\sigma_{0}T}{B^{\frac{\p-1}{\p}}}.
\end{eqnarray*}
Finally, divide both sides of (\ref{eq:fixed-full-info-5}) by $T$
to obtain the desired result.
\end{proof}

\begin{thm}
\label{thm:fixed-partial-info}(Full version of Theorem \ref{thm:main-fixed-partial-info})
Under Assumptions \ref{assu:lb}, \ref{assu:smooth}, \ref{assu:unbias}
and \ref{assu:general-heavy-noise}, let $\Delta_{1}\defeq F(\bx_{1})-F_{*}$,
then for any $T\in\N$, by taking
\begin{eqnarray*}
\beta_{t}\equiv\beta=1-\frac{1}{T^{\frac{1}{2}}}, & \eta_{t}\equiv\eta=\min\left\{ \frac{1}{T^{\frac{3}{4}}},\frac{1}{8L_{1}T^{\frac{1}{2}}}\right\} , & B=\max\left\{ \left\lceil (16\sqrt{2}\sigma_{1})^{\frac{\p}{\p-1}}\right\rceil ,1\right\} ,
\end{eqnarray*}
Algorithm \ref{alg:NSGDM} guarantees
\[
\frac{1}{T}\sum_{t=1}^{T}\E\left[\left\Vert \nabla F(\bx_{t})\right\Vert \right]=\O\left(\frac{\Delta_{1}L_{1}+(\sigma_{0}+\sigma_{1}\left\Vert \nabla F(\bx_{1})\right\Vert )/B^{\frac{\p-1}{\p}}}{\sqrt{T}}+\frac{\Delta_{1}+L_{0}}{T^{\frac{1}{4}}}+\frac{\sigma_{0}}{(BT)^{\frac{\p-1}{\p}}}\right).
\]
In particular, if $\sigma_{1}\leq\frac{1}{16\sqrt{2}}$, we can always
set $B=1$ to get rid of the tail index $\p$.
\end{thm}

\begin{rem}
We note that it is possible to improve the threshold $\frac{1}{16\sqrt{2}}$
to a slightly bigger constant, but this still cannot help us to remove
the requirement of needing $\p$ when $\sigma_{1}$ becomes larger.
Thus, we do not put further effort into optimizing this constant but
keep it as it is.
\end{rem}

\begin{proof}
Note that (\ref{eq:fixed-full-info-4}) still holds under current
parameter choices since $\eta=\min\left\{ \frac{1}{T^{\frac{3}{4}}},\frac{1}{8L_{1}T^{\frac{1}{2}}}\right\} =\min\left\{ \sqrt{\frac{1-\beta}{T}},\frac{1-\beta}{8L_{1}}\right\} $.
Hence, we have
\begin{align*}
 & \sum_{t=1}^{T}\E\left[\left\Vert \nabla F(\bx_{t})\right\Vert \right]\\
\leq & \O\left(\frac{\Delta_{1}}{\eta}+\frac{\eta L_{0}T}{1-\beta}+\frac{1}{B^{\frac{\p-1}{\p}}}\left[\frac{\sigma_{0}+\sigma_{1}\left\Vert \nabla F(\bx_{1})\right\Vert }{1-\beta}+(1-\beta)^{\frac{\p-1}{\p}}T\sigma_{0}\right]\right)\\
\overset{(a)}{=} & \O\left(\frac{\Delta_{1}L_{1}+(\sigma_{0}+\sigma_{1}\left\Vert \nabla F(\bx_{1})\right\Vert )/B^{\frac{\p-1}{\p}}}{1-\beta}+(\Delta_{1}+L_{0})\sqrt{\frac{T}{1-\beta}}+\frac{(1-\beta)^{\frac{\p-1}{\p}}\sigma_{0}T}{B^{\frac{\p-1}{\p}}}\right)\\
\overset{(b)}{=} & \O\left(\left[\Delta_{1}L_{1}+(\sigma_{0}+\sigma_{1}\left\Vert \nabla F(\bx_{1})\right\Vert )/B^{\frac{\p-1}{\p}}\right]\sqrt{T}+(\Delta_{1}+L_{0})T^{\frac{3}{4}}+\frac{\sigma_{0}T^{\frac{\p+1}{2\p}}}{B^{\frac{\p-1}{\p}}}\right),
\end{align*}
where we use $\eta=\min\left\{ \frac{1}{T^{\frac{3}{4}}},\frac{1}{8L_{1}T^{\frac{1}{2}}}\right\} =\min\left\{ \sqrt{\frac{1-\beta}{T}},\frac{1-\beta}{8L_{1}}\right\} $
in $(a)$ and $1-\beta=\frac{1}{T^{\frac{1}{2}}}$ in $(b)$. We divide
both sides by $T$ to obtain the desired result.
\end{proof}

\subsection{Lower Bound}

In this subsection, we will prove the lower bound, Theorem \ref{thm:main-lower-bound}.
The proof is a simple variation of \citet{zhang2020adaptive}, which
itself is based on \citet{carmon2020lower,arjevani2023lower}.

\begin{proof}[Proof of Theorem \ref{thm:main-lower-bound}]
For any $\bx\in\R^{d}$ and $\alpha\in\left[0,1\right]$, we denote
by $\prog_{\alpha}(\bx)$ the highest index whose entry is $\alpha$-far
from $0$, i.e.,
\[
\prog_{\alpha}(\bx)\defeq\max\left\{ i\in\left[d\right]:\left|\bx[i]\right|>\alpha\right\} \text{ where }\max\emptyset\defeq0.
\]
Now given $d\in\N$, we introduce the following underlying function
originally proposed by \citet{carmon2020lower}.
\[
f_{d}(\bx)\defeq-\Psi(1)\Phi(\bx[1])+\sum_{i=2}^{d}\Psi(-\bx[i-1])\Phi(-\bx[i])-\Psi(\bx[i-1])\Phi(\bx[i]),
\]
where 
\[
\Psi(t)\defeq\begin{cases}
0 & t\leq\frac{1}{2}\\
\exp(1-(2t-1)^{-2}) & t>\frac{1}{2}
\end{cases},\quad\Phi(t)\defeq\sqrt{e}\int_{-\infty}^{t}\exp(-\tau^{2}/2)\d\tau.
\]
By Lemma 2 in \citet{arjevani2023lower}, $f_{d}$ admits the following
properties:
\begin{enumerate}
\item $f_{d}(\bzero)-f_{d,*}\leq\delta d$, where $f_{d,*}\defeq\inf_{\bx\in\R^{d}}f_{d}(\bx)$
and $\delta=12$.
\item $f_{d}$ is $\ell$-smooth, where $\ell=152$.
\item For all $\bx\in\R^{d}$, $\left\Vert \nabla f_{d}(\bx)\right\Vert _{\infty}\leq\gamma$,
where $\gamma=23$.
\item For all $\bx\in\R^{d}$, $\prog_{0}(\nabla f_{d}(\bx))\leq\prog_{\frac{1}{2}}(\bx)+1$.
\item For all $\bx\in\R^{d}$ and $i\defeq\prog_{\frac{1}{2}}(\bx)$, $\nabla f_{d}(\bx)=\nabla f_{d}(\bx_{\leq1+i})$
and $\left[\nabla f_{d}(\bx)\right]_{\leq i}=\left[\nabla f_{d}(\bx_{\leq i})\right]_{\leq i}$,
where $\by_{\leq i}[j]\defeq\by[j]\1[j\leq i]$.
\item For all $\bx\in\R^{d}$, if $\prog_{1}(\bx)<d$, then $\left\Vert \nabla f_{d}(\bx)\right\Vert >1$.
\end{enumerate}
Now we consider the following stochastic oracle introduced by \citet{arjevani2023lower},
\[
\boldsymbol{h}_{d}(\bx,z)[i]\defeq\nabla f_{d}(\bx)[i]\left(1+\1\left[i>\prog_{\frac{1}{4}}(\bx)\right]\left(\frac{z}{q}-1\right)\right),\forall i\in\left[d\right],
\]
where $z=\mathrm{Bernoulli}(q)$ and $q\in\left[0,1\right]$ will
be specified later. As one can check, $\E_{z}\left[\boldsymbol{h}_{d}(\bx,z)\right]=\nabla f_{d}(\bx),\forall\bx\in\R^{d}$.
Moreover, by Lemma 3 in \citet{arjevani2023lower}, we know $\boldsymbol{h}_{d}$
is a probability-$q$ zero-chain (see Definition 2 in \citet{arjevani2023lower}
for what it is) and satisfies almost surely
\begin{equation}
\left\Vert \boldsymbol{h}_{d}(\bx,z)-\nabla f_{d}(\bx)\right\Vert \leq\gamma\left|\frac{z}{q}-1\right|,\forall\bx\in\R^{d}.\label{eq:lower-1}
\end{equation}

Next, given $\p\in\left(1,2\right]$, $\Delta_{1},L_{0},\sigma_{0}>0$,
and small enough $\varepsilon$, we define $d\defeq\left\lfloor \frac{\Delta_{1}L_{0}}{4\delta\ell\varepsilon^{2}}\right\rfloor $
and
\[
F_{d}(\bx)\defeq\frac{L_{0}\lambda^{2}}{\ell}f_{d}\left(\frac{\bx}{\lambda}\right),\text{ where }\lambda\defeq\frac{2\ell\varepsilon}{L_{0}}.
\]
Moreover, we let
\[
\bg_{d}(\bx,z)\defeq\frac{L_{0}\lambda}{\ell}\boldsymbol{h}_{d}\left(\frac{\bx}{\lambda},z\right),\text{ and }q\defeq\left(\frac{4\gamma\varepsilon}{\sigma_{0}}\right)^{\frac{\p}{\p-1}}.
\]
$q\leq1$ can be true since we assume that $\varepsilon$ is small
enough. 

Note that $F_{d}$ is lower bounded since $f_{d}$ is lower bounded
and thus satisfies Assumption \ref{assu:lb}. In addition, we have
$\nabla F_{d}(\bx)=\frac{L_{0}\lambda}{\ell}\nabla f_{d}\left(\frac{\bx}{\lambda}\right)$,
which implies $F_{d}$ is $L_{0}$-smooth because $f_{d}$ is $\ell$-smooth.
So $F_{d}$ also satisfies Assumption \ref{assu:smooth} with $L_{1}=0$.
Moreover, we can find
\[
F_{d}(\bzero)-\inf_{\bx\in\R^{d}}F_{d}(\bx)=\frac{L_{0}\lambda^{2}}{\ell}\left(f_{d}\left(\frac{\bx}{\lambda}\right)-f_{d,*}\right)\leq\frac{L_{0}\lambda^{2}}{\ell}\cdot\delta d\leq\frac{L_{0}\lambda^{2}}{\ell}\cdot\delta\cdot\frac{\Delta_{1}L_{0}}{4\delta\ell\varepsilon^{2}}=\Delta_{1}.
\]
Now let us verify
\[
\E_{z}\left[\bg_{d}(\bx,z)\right]=\E_{z}\left[\frac{L_{0}\lambda}{\ell}\boldsymbol{h}_{d}\left(\frac{\bx}{\lambda},z\right)\right]=\frac{L_{0}\lambda}{\ell}\nabla f_{d}\left(\frac{\bx}{\lambda}\right)=\nabla F_{d}(\bx).
\]
Hence, $\bg_{d}(\bx,z)$ satisfies Assumption \ref{assu:unbias}.
Lastly, we know
\begin{align*}
\E_{z}\left[\left\Vert \bg_{d}(\bx,z)-\nabla F_{d}(\bx)\right\Vert ^{\p}\right] & =\left(\frac{L_{0}\lambda}{\ell}\right)^{\p}\E_{z}\left[\left\Vert \boldsymbol{h}_{d}\left(\frac{\bx}{\lambda},z\right)-\nabla f_{d}\left(\frac{\bx}{\lambda}\right)\right\Vert ^{\p}\right]\overset{\eqref{eq:lower-1}}{\leq}\left(\frac{L_{0}\lambda\gamma}{\ell}\right)^{\p}\E\left[\left|\frac{z}{q}-1\right|^{\p}\right]\\
 & =\left(\frac{L_{0}\lambda\gamma}{\ell}\right)^{\p}\left(1-q+\frac{(1-q)^{\p}}{q^{\p-1}}\right)=(2\gamma\varepsilon)^{\p}(1-q)\frac{q^{\p-1}+(1-q)^{\p-1}}{q^{\p-1}}\\
 & \leq\frac{(4\gamma\varepsilon)^{\p}}{q^{\p-1}}=\sigma_{0}^{\p}.
\end{align*}
Thus, $\bg_{d}(\bx,z)$ satisfies Assumption \ref{assu:general-heavy-noise}
with $\sigma_{1}=0$.

Finally, by Lemma 1 in \citet{arjevani2023lower}, for any zero-respecting
algorithm starting from $\bx_{1}=\bzero$, with probability at least
$\frac{1}{2}$, $\prog_{0}(\bx_{t})<d,\forall t\leq\frac{d-1}{2q}$.
By noticing that $\prog_{1}\left(\frac{\bx_{t}}{\lambda}\right)\leq\prog_{0}\left(\frac{\bx_{t}}{\lambda}\right)=\prog_{0}(\bx_{t})<d$,
we then have with probability at least $\frac{1}{2}$, $\left\Vert \nabla f_{d}\left(\frac{\bx_{t}}{\lambda}\right)\right\Vert >1,\forall t\leq\frac{d-1}{2q}$,
which implies
\[
\E\left[\left\Vert \nabla F_{d}(\bx_{t})\right\Vert \right]=\frac{L_{0}\lambda}{\ell}\E\left[\left\Vert \nabla f_{d}\left(\frac{\bx_{t}}{\lambda}\right)\right\Vert \right]>\frac{L_{0}\lambda}{2\ell}=\varepsilon,\forall t\leq\frac{d-1}{2q}.
\]
Thus, to output an $\varepsilon$-stationary point, the algorithm
need at least the following number of iterations
\[
\frac{d-1}{2q}=\frac{1}{2}\left(\left\lfloor \frac{\Delta_{1}L_{0}}{4\delta\ell\varepsilon^{2}}\right\rfloor -1\right)\cdot\left(\frac{\sigma_{0}}{4\gamma\varepsilon}\right)^{\frac{\p}{\p-1}}=\Omega\left(\Delta_{1}L_{0}\sigma_{0}^{\frac{\p}{\p-1}}\varepsilon^{-\frac{3\p-2}{\p-1}}\right).
\]
\end{proof}

\subsection{Helpful Lemmas}

We first recall the notations as follows

\begin{align}
\bep_{t} & \defeq\begin{cases}
\bm{_{t}}-\nabla F(\bx_{t}) & t\in\left[T\right]\\
\bm{_{0}}-\nabla F(\bx_{1}) & t=0
\end{cases},\label{eq:def-eps}\\
\bD_{t} & \defeq\1_{t\geq2}\left(\nabla F(\bx_{t-1})-\nabla F(\bx_{t})\right),\forall t\in\left[T\right],\label{eq:def-D}\\
\bxi_{t}^{i} & \defeq\bg_{t}^{i}-\nabla F(\bx_{t}),\forall i\in\left[B\right],\forall t\in\left[T\right],\label{eq:def-xi-i}\\
\bxi_{t} & \defeq\bg_{t}-\nabla F(\bx_{t})=\frac{1}{B}\sum_{i=1}^{B}\bxi_{t}^{i},\forall t\in\left[T\right].\label{eq:def-xi}
\end{align}
Now we are ready to start the proof.
\begin{lem}
\label{lem:descent-lemma}(Full version of Lemma \ref{lem:main-descent-lemma})
Under Assumptions \ref{assu:lb} and \ref{assu:smooth}, let $\Delta_{1}\defeq F(\bx_{1})-F_{*}$,
if $\eta_{t}\leq\frac{1}{L_{1}},\forall t\in\left[T\right]$, then
Algorithm \ref{alg:NSGDM} guarantees
\[
\sum_{t=1}^{T}\eta_{t}\E\left[\left\Vert \nabla F(\bx_{t})\right\Vert \right]\leq\Delta_{1}+\sum_{t=1}^{T}2\eta_{t}\E\left[\left\Vert \bep_{t}\right\Vert \right]+\sum_{t=1}^{T}\frac{L_{0}+L_{1}\E\left[\left\Vert \nabla F(\bx_{t})\right\Vert \right]}{2}\eta_{t}^{2}.
\]
\end{lem}

\begin{proof}
Note that $\left\Vert \bx_{t+1}-\bx_{t}\right\Vert =\left\Vert \eta_{t}\frac{\bm{_{t}}}{\left\Vert \bm{_{t}}\right\Vert }\right\Vert \leq\eta_{t}\leq\frac{1}{L_{1}}$
now, we then invoke Lemma (\ref{lem:smooth-ineq}) to obtain for any
$t\in\left[T\right]$,
\begin{align}
F(\bx_{t+1}) & \leq F(\bx_{t})+\left\langle \nabla F(\bx_{t}),\bx_{t+1}-\bx_{t}\right\rangle +\frac{L_{0}+L_{1}\left\Vert \nabla F(\bx_{t})\right\Vert }{2}\left\Vert \bx_{t+1}-\bx_{t}\right\Vert ^{2}\nonumber \\
 & =F(\bx_{t})-\eta_{t}\left\langle \nabla F(\bx_{t}),\nicefrac{\bm{_{t}}}{\left\Vert \bm{_{t}}\right\Vert }\right\rangle +\frac{L_{0}+L_{1}\left\Vert \nabla F(\bx_{t})\right\Vert }{2}\eta_{t}^{2}\nonumber \\
 & \overset{\eqref{eq:def-eps}}{=}F(\bx_{t})-\eta_{t}\left\Vert \bm{_{t}}\right\Vert +\eta_{t}\left\langle \bep_{t},\nicefrac{\bm{_{t}}}{\left\Vert \bm{_{t}}\right\Vert }\right\rangle +\frac{L_{0}+L_{1}\left\Vert \nabla F(\bx_{t})\right\Vert }{2}\eta_{t}^{2}\nonumber \\
 & \overset{(a)}{\leq}F(\bx_{t})-\eta_{t}\left\Vert \bm{_{t}}\right\Vert +\eta_{t}\left\Vert \bep_{t}\right\Vert +\frac{L_{0}+L_{1}\left\Vert \nabla F(\bx_{t})\right\Vert }{2}\eta_{t}^{2}\nonumber \\
 & \overset{(b)}{\leq}F(\bx_{t})-\eta_{t}\left\Vert \nabla F(\bx_{t})\right\Vert +2\eta_{t}\left\Vert \bep_{t}\right\Vert +\frac{L_{0}+L_{1}\left\Vert \nabla F(\bx_{t})\right\Vert }{2}\eta_{t}^{2},\label{eq:descent-lemma-1}
\end{align}
where $(a)$ is by $\left\langle \bep_{t},\nicefrac{\bm{_{t}}}{\left\Vert \bm{_{t}}\right\Vert }\right\rangle \leq\left\Vert \bep_{t}\right\Vert \left\Vert \nicefrac{\bm{_{t}}}{\left\Vert \bm{_{t}}\right\Vert }\right\Vert =\left\Vert \bep_{t}\right\Vert $
and $(b)$ is due to $\left\Vert \bm{_{t}}\right\Vert \overset{\eqref{eq:def-eps}}{=}\left\Vert \nabla F(\bx_{t})+\bep_{t}\right\Vert \geq\left\Vert \nabla F(\bx_{t})\right\Vert -\left\Vert \bep_{t}\right\Vert $.
Taking expectations on both sides of (\ref{eq:descent-lemma-1}) and
summing up from $t=1$ to $T$, we have
\[
\E\left[F(\bx_{T+1})\right]\leq F(\bx_{1})-\sum_{t=1}^{T}\eta_{t}\E\left[\left\Vert \nabla F(\bx_{t})\right\Vert \right]+\sum_{t=1}^{T}2\eta_{t}\E\left[\left\Vert \bep_{t}\right\Vert \right]+\sum_{t=1}^{T}\frac{L_{0}+L_{1}\E\left[\left\Vert \nabla F(\bx_{t})\right\Vert \right]}{2}\eta_{t}^{2}.
\]
Finally, we rearrange the terms, apply $\E\left[F(\bx_{T+1})\right]\geq F_{*}$
due to Assumption \ref{assu:lb} and use $\Delta_{1}=F(\bx_{1})-F_{*}$
to get the desired result.
\end{proof}

\begin{lem}
\label{lem:err-bound}(Full version of Lemma \ref{lem:main-err-bound})
Under Assumptions \ref{assu:smooth}, \ref{assu:unbias} and \ref{assu:general-heavy-noise},
if $\eta_{t}\leq\frac{1}{L_{1}},\forall t\in\left[T\right]$, then
Algorithm \ref{alg:NSGDM} guarantees
\begin{align*}
\E\left[\left\Vert \bep_{t}\right\Vert \right]\leq & \frac{2\sqrt{2}}{B^{\frac{\p-1}{\p}}}\left[\beta_{1:t}\left(\sigma_{0}+\sigma_{1}\left\Vert \nabla F(\bx_{1})\right\Vert \right)+\left(\sum_{s=1}^{t}(1-\beta_{s})^{\p}(\beta_{s+1:t})^{\p}\right)^{\frac{1}{\p}}\sigma_{0}\right]\\
 & +\sum_{s=2}^{t}\beta_{s:t}\left(L_{0}+L_{1}\E\left[\left\Vert \nabla F(\bx_{s-1})\right\Vert \right]\right)\eta_{s-1}\\
 & +\frac{2\sqrt{2}}{B^{\frac{\p-1}{\p}}}\sum_{s=1}^{t}(1-\beta_{s})\beta_{s+1:t}\sigma_{1}\E\left[\left\Vert \nabla F(\bx_{s})\right\Vert \right],\forall t\in\left[T\right].
\end{align*}
\end{lem}

\begin{proof}
Based on Lemma \ref{lem:main-decomposition}, we know for any $t\in\left[T\right]$,
\begin{align}
\left\Vert \bep_{t}\right\Vert  & \leq\beta_{1:t}\left\Vert \bep_{0}\right\Vert +\left\Vert \sum_{s=1}^{t}\beta_{s:t}\bD_{s}\right\Vert +\left\Vert \sum_{s=1}^{t}(1-\beta_{s})\beta_{s+1:t}\bxi_{s}\right\Vert \nonumber \\
\Rightarrow\E\left[\left\Vert \bep_{t}\right\Vert \right] & \leq\beta_{1:t}\E\left[\left\Vert \bep_{0}\right\Vert \right]+\E\left[\left\Vert \sum_{s=1}^{t}\beta_{s:t}\bD_{s}\right\Vert \right]+\E\left[\left\Vert \sum_{s=1}^{t}(1-\beta_{s})\beta_{s+1:t}\bxi_{s}\right\Vert \right].\label{eq:err-bound-1}
\end{align}

First, by the definition of $\bep_{0}\overset{\eqref{eq:def-eps}}{=}\bm{_{0}}-\nabla F(\bx_{1})=\bg_{1}-\nabla F(\bx_{1})\overset{\eqref{eq:def-xi-i}}{=}\frac{1}{B}\sum_{i=1}^{B}\bxi_{1}^{i}$,
we have
\begin{align}
\E\left[\left\Vert \bep_{0}\right\Vert \right] & =\frac{1}{B}\E\left[\left\Vert \sum_{i=1}^{B}\bxi_{1}^{i}\right\Vert \right]\overset{(a)}{\leq}\frac{2\sqrt{2}}{B}\E\left[\left(\sum_{i=1}^{B}\left\Vert \bxi_{1}^{i}\right\Vert ^{\p}\right)^{\frac{1}{\p}}\right]\nonumber \\
 & \overset{(b)}{\leq}\frac{2\sqrt{2}}{B}\left(\sum_{i=1}^{B}\E\left[\left\Vert \bxi_{1}^{i}\right\Vert ^{\p}\right]\right)^{\frac{1}{\p}}\overset{\text{Assumption }\ref{assu:general-heavy-noise}}{\leq}\frac{2\sqrt{2}}{B^{\frac{\p-1}{\p}}}\left(\sigma_{0}^{\p}+\sigma_{1}^{\p}\left\Vert \nabla F(\bx_{1})\right\Vert ^{\p}\right)^{\frac{1}{\p}}\nonumber \\
 & \overset{(c)}{\leq}\frac{2\sqrt{2}}{B^{\frac{\p-1}{\p}}}(\sigma_{0}+\sigma_{1}\left\Vert \nabla F(\bx_{1})\right\Vert ),\label{eq:err-bound-2}
\end{align}
where $(a)$ is by applying Lemma \ref{lem:main-core} with $\bv_{i}\defeq\bxi_{1}^{i},\forall i\in\left[B\right]$,
$(b)$ is due to H\"{o}lder's inequality, and $(c)$ is because of
$(x+y)^{\frac{1}{\p}}\leq x^{\frac{1}{\p}}+y^{\frac{1}{\p}}$ when
$\p\geq1$.

Next, we know 
\begin{align*}
\left\Vert \sum_{s=1}^{t}\beta_{s:t}\bD_{s}\right\Vert  & \leq\sum_{s=1}^{t}\beta_{s:t}\left\Vert \bD_{s}\right\Vert \overset{\eqref{eq:def-D}}{=}\sum_{s=1}^{t}\beta_{s:t}\left\Vert \nabla F(\bx_{s-1})-\nabla F(\bx_{s})\right\Vert \1_{s\geq2}\\
 & \overset{\text{Assumption }\ref{assu:smooth}}{\leq}\sum_{s=1}^{t}\beta_{s:t}\left(L_{0}+L_{1}\left\Vert \nabla F(\bx_{s-1})\right\Vert \right)\left\Vert \bx_{s}-\bx_{s-1}\right\Vert \1_{s\geq2}\\
 & \overset{(d)}{\leq}\sum_{s=1}^{t}\beta_{s:t}\left(L_{0}+L_{1}\left\Vert \nabla F(\bx_{s-1})\right\Vert \right)\eta_{s-1}\1_{s\geq2}\\
 & =\sum_{s=2}^{t}\beta_{s:t}\left(L_{0}+L_{1}\left\Vert \nabla F(\bx_{s-1})\right\Vert \right)\eta_{s-1},
\end{align*}
where $(d)$ is by $\left\Vert \bx_{s}-\bx_{s-1}\right\Vert =\left\Vert \eta_{s-1}\frac{\bm{_{s-1}}}{\left\Vert \bm{_{s-1}}\right\Vert }\right\Vert \leq\eta_{s-1}$
from the update rule of Algorithm \ref{alg:NSGDM}. Therefore, we
have
\begin{equation}
\E\left[\left\Vert \sum_{s=1}^{t}\beta_{s:t}\bD_{s}\right\Vert \right]\leq\sum_{s=2}^{t}\beta_{s:t}\left(L_{0}+L_{1}\E\left[\left\Vert \nabla F(\bx_{s-1})\right\Vert \right]\right)\eta_{s-1}.\label{eq:err-bound-3}
\end{equation}

Moreover, let $\bv_{(s-1)B+i}\defeq(1-\beta_{s})\beta_{s+1:t}\bxi_{s}^{i},\forall i\in\left[B\right],s\in\left[t\right]$
and note that this sequence satisfies the requirement of Lemma \ref{lem:main-core},
then there is
\begin{align}
\E\left[\left\Vert \sum_{s=1}^{t}(1-\beta_{s})\beta_{s+1:t}\bxi_{s}\right\Vert \right] & \overset{\eqref{eq:def-xi}}{=}\frac{1}{B}\E\left[\left\Vert \sum_{s=1}^{t}\sum_{i=1}^{B}(1-\beta_{s})\beta_{s+1:t}\bxi_{s}^{i}\right\Vert \right]=\frac{1}{B}\E\left[\left\Vert \sum_{s=1}^{t}\sum_{i=1}^{B}\bv_{(s-1)B+i}\right\Vert \right]\nonumber \\
 & \leq\frac{2\sqrt{2}}{B}\E\left[\left(\sum_{s=1}^{t}\sum_{i=1}^{B}\left\Vert \bv_{(s-1)B+i}\right\Vert ^{\p}\right)^{\frac{1}{\p}}\right]\nonumber \\
 & =\frac{2\sqrt{2}}{B}\E\left[\left(\sum_{s=1}^{t}\sum_{i=1}^{B}(1-\beta_{s})^{\p}(\beta_{s+1:t})^{\p}\left\Vert \bxi_{s}^{i}\right\Vert ^{\p}\right)^{\frac{1}{\p}}\right].\label{eq:err-bound-4}
\end{align}
Observe that
\begin{align}
 & \E\left[\left(\sum_{s=1}^{t}\sum_{i=1}^{B}(1-\beta_{s})^{\p}(\beta_{s+1:t})^{\p}\left\Vert \bxi_{s}^{i}\right\Vert ^{\p}\right)^{\frac{1}{\p}}\mid\F_{t-1}\right]\nonumber \\
\overset{(e)}{\leq} & \left(\E\left[\sum_{s=1}^{t}\sum_{i=1}^{B}(1-\beta_{s})^{\p}(\beta_{s+1:t})^{\p}\left\Vert \bxi_{s}^{i}\right\Vert ^{\p}\right]\mid\F_{t-1}\right)^{\frac{1}{\p}}\nonumber \\
= & \left(\E\left[\sum_{i=1}^{B}(1-\beta_{t})^{\p}(\beta_{t+1:t})^{\p}\left\Vert \bxi_{t}^{i}\right\Vert ^{\p}\mid\F_{t-1}\right]+\sum_{s=1}^{t-1}\sum_{i=1}^{B}(1-\beta_{s})^{\p}(\beta_{s+1:t})^{\p}\left\Vert \bxi_{s}^{i}\right\Vert ^{\p}\right)^{\frac{1}{\p}}\nonumber \\
\overset{\text{Assumption }\ref{assu:general-heavy-noise}}{\leq} & \left(B(1-\beta_{t})^{\p}(\beta_{t+1:t})^{\p}(\sigma_{0}^{\p}+\sigma_{1}^{\p}\left\Vert \nabla F(\bx_{t})\right\Vert ^{\p})+\sum_{s=1}^{t-1}\sum_{i=1}^{B}(1-\beta_{s})^{\p}(\beta_{s+1:t})^{\p}\left\Vert \bxi_{s}^{i}\right\Vert ^{\p}\right)^{\frac{1}{\p}}\nonumber \\
\leq & \left(B(1-\beta_{t})^{\p}(\beta_{t+1:t})^{\p}\sigma_{0}^{\p}+\sum_{s=1}^{t-1}\sum_{i=1}^{B}(1-\beta_{s})^{\p}(\beta_{s+1:t})^{\p}\left\Vert \bxi_{s}^{i}\right\Vert ^{\p}\right)^{\frac{1}{\p}}\nonumber \\
 & +B^{\frac{1}{\p}}(1-\beta_{t})\beta_{t+1:t}\sigma_{1}\left\Vert \nabla F(\bx_{t})\right\Vert ,\label{eq:err-bound-5}
\end{align}
where $(e)$ is by H\"{o}lder's inequality. Taking expectations on
both sides of (\ref{eq:err-bound-5}) to get
\begin{align*}
 & \E\left[\left(\sum_{s=1}^{t}\sum_{i=1}^{B}(1-\beta_{s})^{\p}(\beta_{s+1:t})^{\p}\left\Vert \bxi_{s}^{i}\right\Vert ^{\p}\right)^{\frac{1}{\p}}\right]\\
\leq & \E\left[\left(B(1-\beta_{t})^{\p}(\beta_{t+1:t})^{\p}\sigma_{0}^{\p}+\sum_{s=1}^{t-1}\sum_{i=1}^{B}(1-\beta_{s})^{\p}(\beta_{s+1:t})^{\p}\left\Vert \bxi_{s}^{i}\right\Vert ^{\p}\right)^{\frac{1}{\p}}\right]\\
 & +B^{\frac{1}{\p}}(1-\beta_{t})\beta_{t+1:t}\sigma_{1}\E\left[\left\Vert \nabla F(\bx_{t})\right\Vert \right].
\end{align*}
Recursively applying the above argument from $\F_{t-2}$ to $\F_{0}$,
we can finally obtain
\begin{align}
 & \E\left[\left(\sum_{s=1}^{t}\sum_{i=1}^{B}(1-\beta_{s})^{\p}(\beta_{s+1:t})^{\p}\left\Vert \bxi_{s}^{i}\right\Vert ^{\p}\right)^{\frac{1}{\p}}\right]\nonumber \\
\leq & B^{\frac{1}{\p}}\left(\sum_{s=1}^{t}(1-\beta_{s})^{\p}(\beta_{s+1:t})^{\p}\right)^{\frac{1}{\p}}\sigma_{0}+B^{\frac{1}{\p}}\sum_{s=1}^{t}(1-\beta_{s})\beta_{s+1:t}\sigma_{1}\E\left[\left\Vert \nabla F(\bx_{s})\right\Vert \right],\label{eq:err-bound-6}
\end{align}
which gives us
\begin{align}
 & \E\left[\left\Vert \sum_{s=1}^{t}(1-\beta_{s})\beta_{s+1:t}\bxi_{s}\right\Vert \right]\nonumber \\
\overset{\eqref{eq:err-bound-4}}{\leq} & \frac{2\sqrt{2}}{B}\E\left[\left(\sum_{s=1}^{t}\sum_{i=1}^{B}(1-\beta_{s})^{\p}(\beta_{s+1:t})^{\p}\left\Vert \bxi_{s}^{i}\right\Vert ^{\p}\right)^{\frac{1}{\p}}\right]\nonumber \\
\overset{\eqref{eq:err-bound-6}}{\leq} & \frac{2\sqrt{2}}{B^{\frac{\p-1}{\p}}}\left(\sum_{s=1}^{t}(1-\beta_{s})^{\p}(\beta_{s+1:t})^{\p}\right)^{\frac{1}{\p}}\sigma_{0}+\frac{2\sqrt{2}}{B^{\frac{\p-1}{\p}}}\sum_{s=1}^{t}(1-\beta_{s})\beta_{s+1:t}\sigma_{1}\E\left[\left\Vert \nabla F(\bx_{s})\right\Vert \right].\label{eq:err-bound-7}
\end{align}

Combining (\ref{eq:err-bound-1}), (\ref{eq:err-bound-2}), (\ref{eq:err-bound-3})
and (\ref{eq:err-bound-7}), we finally obtain for any $t\in\left[T\right]$,
\begin{align*}
\E\left[\left\Vert \bep_{t}\right\Vert \right]\leq & \frac{2\sqrt{2}}{B^{\frac{\p-1}{\p}}}\left[\beta_{1:t}\left(\sigma_{0}+\sigma_{1}\left\Vert \nabla F(\bx_{1})\right\Vert \right)+\left(\sum_{s=1}^{t}(1-\beta_{s})^{\p}(\beta_{s+1:t})^{\p}\right)^{\frac{1}{\p}}\sigma_{0}\right]\\
 & +\sum_{s=2}^{t}\beta_{s:t}\left(L_{0}+L_{1}\E\left[\left\Vert \nabla F(\bx_{s-1})\right\Vert \right]\right)\eta_{s-1}\\
 & +\frac{2\sqrt{2}}{B^{\frac{\p-1}{\p}}}\sum_{s=1}^{t}(1-\beta_{s})\beta_{s+1:t}\sigma_{1}\E\left[\left\Vert \nabla F(\bx_{s})\right\Vert \right].
\end{align*}
\end{proof}

\end{document}